\documentclass[a4paper,11pt]{amsart}    
\usepackage{latexsym, amsmath, amsthm, amssymb, setspace,verbatim}
\usepackage[all]{xy}
\usepackage{longtable}
\usepackage{hyperref, aliascnt}
\usepackage{enumitem}

\title{Equivariant property (SI) revisited}
\author{G{\'a}bor Szab{\'o}}

\address{KU Leuven, Department of Mathematics, Celestijnenlaan 200b box 2400, \phantom{-..-}B-3001 Leuven, Belgium}
\email{gabor.szabo@kuleuven.be}
\thanks{This work was supported by the start-up grant STG/18/019 of KU Leuven, the research project C14/19/088 funded by the research council of KU Leuven.}

\begin{document}

\newcommand{\IA}[0]{\mathbb{A}} \newcommand{\IB}[0]{\mathbb{B}}
\newcommand{\IC}[0]{\mathbb{C}} \newcommand{\ID}[0]{\mathbb{D}}
\newcommand{\IE}[0]{\mathbb{E}} \newcommand{\IF}[0]{\mathbb{F}}
\newcommand{\IG}[0]{\mathbb{G}} \newcommand{\IH}[0]{\mathbb{H}}
\newcommand{\II}[0]{\mathbb{I}} \renewcommand{\IJ}[0]{\mathbb{J}}
\newcommand{\IK}[0]{\mathbb{K}} \newcommand{\IL}[0]{\mathbb{L}}
\newcommand{\IM}[0]{\mathbb{M}} \newcommand{\IN}[0]{\mathbb{N}}
\newcommand{\IO}[0]{\mathbb{O}} \newcommand{\IP}[0]{\mathbb{P}}
\newcommand{\IQ}[0]{\mathbb{Q}} \newcommand{\IR}[0]{\mathbb{R}}
\newcommand{\IS}[0]{\mathbb{S}} \newcommand{\IT}[0]{\mathbb{T}}
\newcommand{\IU}[0]{\mathbb{U}} \newcommand{\IV}[0]{\mathbb{V}}
\newcommand{\IW}[0]{\mathbb{W}} \newcommand{\IX}[0]{\mathbb{X}}
\newcommand{\IY}[0]{\mathbb{Y}} \newcommand{\IZ}[0]{\mathbb{Z}}

\newcommand{\CA}[0]{\mathcal{A}} \newcommand{\CB}[0]{\mathcal{B}}
\newcommand{\CC}[0]{\mathcal{C}} \newcommand{\CD}[0]{\mathcal{D}}
\newcommand{\CE}[0]{\mathcal{E}} \newcommand{\CF}[0]{\mathcal{F}}
\newcommand{\CG}[0]{\mathcal{G}} \newcommand{\CH}[0]{\mathcal{H}}
\newcommand{\CI}[0]{\mathcal{I}} \newcommand{\CJ}[0]{\mathcal{J}}
\newcommand{\CK}[0]{\mathcal{K}} \newcommand{\CL}[0]{\mathcal{L}}
\newcommand{\CM}[0]{\mathcal{M}} \newcommand{\CN}[0]{\mathcal{N}}
\newcommand{\CO}[0]{\mathcal{O}} \newcommand{\CP}[0]{\mathcal{P}}
\newcommand{\CQ}[0]{\mathcal{Q}} \newcommand{\CR}[0]{\mathcal{R}}
\newcommand{\CS}[0]{\mathcal{S}} \newcommand{\CT}[0]{\mathcal{T}}
\newcommand{\CU}[0]{\mathcal{U}} \newcommand{\CV}[0]{\mathcal{V}}
\newcommand{\CW}[0]{\mathcal{W}} \newcommand{\CX}[0]{\mathcal{X}}
\newcommand{\CY}[0]{\mathcal{Y}} \newcommand{\CZ}[0]{\mathcal{Z}}

\newcommand{\FA}[0]{\mathfrak{A}} \newcommand{\FB}[0]{\mathfrak{B}}
\newcommand{\FC}[0]{\mathfrak{C}} \newcommand{\FD}[0]{\mathfrak{D}}
\newcommand{\FE}[0]{\mathfrak{E}} \newcommand{\FF}[0]{\mathfrak{F}}
\newcommand{\FG}[0]{\mathfrak{G}} \newcommand{\FH}[0]{\mathfrak{H}}
\newcommand{\FI}[0]{\mathfrak{I}} \newcommand{\FJ}[0]{\mathfrak{J}}
\newcommand{\FK}[0]{\mathfrak{K}} \newcommand{\FL}[0]{\mathfrak{L}}
\newcommand{\FM}[0]{\mathfrak{M}} \newcommand{\FN}[0]{\mathfrak{N}}
\newcommand{\FO}[0]{\mathfrak{O}} \newcommand{\FP}[0]{\mathfrak{P}}
\newcommand{\FQ}[0]{\mathfrak{Q}} \newcommand{\FR}[0]{\mathfrak{R}}
\newcommand{\FS}[0]{\mathfrak{S}} \newcommand{\FT}[0]{\mathfrak{T}}
\newcommand{\FU}[0]{\mathfrak{U}} \newcommand{\FV}[0]{\mathfrak{V}}
\newcommand{\FW}[0]{\mathfrak{W}} \newcommand{\FX}[0]{\mathfrak{X}}
\newcommand{\FY}[0]{\mathfrak{Y}} \newcommand{\FZ}[0]{\mathfrak{Z}}

\newcommand\set[1]{\left\{#1\right\}}  
\renewcommand{\phi}[0]{\varphi}
\newcommand{\eps}[0]{\varepsilon}
\newcommand{\cstar}[0]{\ensuremath{\mathrm{C}^*}}
\newcommand{\cc}[0]{\ensuremath{\simeq_{\mathrm{cc}}}}
\newcommand{\id}[0]{\ensuremath{\operatorname{id}}}
\newcommand{\dist}[0]{\ensuremath{\operatorname{dist}}}
\newcommand{\dst}[0]{\displaystyle}
\newcommand{\dimrokc}[0]{\dim_{\mathrm{Rok}}^{\mathrm{c}}}

\newtheorem{satz}{Satz}[section]		

\newaliascnt{corCT}{satz}
\newtheorem{corollary}[corCT]{Corollary}
\aliascntresetthe{corCT}
\providecommand*{\corCTautorefname}{Corollary}
\newaliascnt{lemmaCT}{satz}
\newtheorem{lemma}[lemmaCT]{Lemma}
\aliascntresetthe{lemmaCT}
\providecommand*{\lemmaCTautorefname}{Lemma}
\newaliascnt{propCT}{satz}
\newtheorem{proposition}[propCT]{Proposition}
\aliascntresetthe{propCT}
\providecommand*{\propCTautorefname}{Proposition}
\newaliascnt{theoremCT}{satz}
\newtheorem{theorem}[theoremCT]{Theorem}
\aliascntresetthe{theoremCT}
\providecommand*{\theoremCTautorefname}{Theorem}
\newtheorem*{theoreme}{Theorem}
\newtheorem*{core}{Corollary}

\theoremstyle{definition}

\newaliascnt{conjectureCT}{satz}
\newtheorem{conjecture}[conjectureCT]{Conjecture}
\aliascntresetthe{conjectureCT}
\providecommand*{\conjectureCTautorefname}{Conjecture}
\newaliascnt{defiCT}{satz}
\newtheorem{definition}[defiCT]{Definition}
\aliascntresetthe{defiCT}
\providecommand*{\defiCTautorefname}{Definition}
\newtheorem*{defie}{Definition}
\newaliascnt{notaCT}{satz}
\newtheorem{notation}[notaCT]{Notation}
\aliascntresetthe{notaCT}
\providecommand*{\notaCTautorefname}{Notation}
\newtheorem*{notae}{Notation}
\newaliascnt{remCT}{satz}
\newtheorem{remark}[remCT]{Remark}
\aliascntresetthe{remCT}
\providecommand*{\remCTautorefname}{Remark}
\newtheorem*{reme}{Remark}
\newaliascnt{exampleCT}{satz}
\newtheorem{example}[exampleCT]{Example}
\aliascntresetthe{exampleCT}
\providecommand*{\exampleCTautorefname}{Example}
\newaliascnt{questionCT}{satz}
\newtheorem{question}[questionCT]{Question}
\aliascntresetthe{questionCT}
\providecommand*{\questionCTautorefname}{Question}

\newcounter{theoremintro}
\renewcommand*{\thetheoremintro}{\Alph{theoremintro}}
\newaliascnt{theoremiCT}{theoremintro}
\newtheorem{theoremi}[theoremiCT]{Theorem}
\aliascntresetthe{theoremiCT}
\providecommand*{\theoremiCTautorefname}{Theorem}
\newaliascnt{coriCT}{theoremintro}
\newtheorem{cori}[coriCT]{Corollary}
\aliascntresetthe{coriCT}
\providecommand*{\coriCTautorefname}{Corollary}
\newaliascnt{conjectureiCT}{theoremintro}
\newtheorem{conjecturei}[conjectureiCT]{Conjecture}
\aliascntresetthe{conjectureiCT}
\providecommand*{\conjectureiCTautorefname}{Conjecture}

\numberwithin{equation}{section}
\renewcommand{\theequation}{e\thesection.\arabic{equation}}


\begin{abstract}
We revisit Matui--Sato's notion of property (SI) for \cstar-algebras and \cstar-dynamics.
More specifically, we generalize the known framework to the case of \cstar-algebras with possibly unbounded traces.
The novelty of this approach lies in the equivariant context, where none of the previous work allows one to (directly) apply such methods to actions of amenable groups on highly non-unital \cstar-algebras, in particular to establish equivariant Jiang--Su stability.
Our main result is an extension of an observation by Sato: For any countable amenable group $\Gamma$ and any non-elementary separable simple nuclear \cstar-algebra $A$ with strict comparison, every $\Gamma$-action on $A$ has equivariant property (SI).
A more general statement involving relative property (SI) for inclusions into ultraproducts is proved as well.
As a consequence we show that if $A$ also has finitely many rays of extremal traces, then every $\Gamma$-action on $A$ is equivariantly Jiang--Su stable.
We moreover provide applications of the main result to the context of strongly outer actions, such as a generalization of Nawata's classification of strongly outer automorphisms on the (stabilized) Razak--Jacelon algebra.
\end{abstract}

\maketitle

\tableofcontents


\section*{Introduction}

This paper is part of an ongoing effort to unravel the fine structure for actions of amenable groups on simple nuclear \cstar-algebras that fall within the scope of the Elliott program.
Our understanding about such \cstar-algebras and their classification has improved at a fast and furious pace in recent years --- see for example \cite{GongLinNiu15, ElliottNiu15, ElliottGongLinNiu15, TikuisisWhiteWinter17, GongLinNiu18, ElliottGongLinNiu17, GongLin17, Winter17} --- which naturally calls for an investigation into the structure of their innate symmetries, i.e., the ways in which groups can act on them.
The historical precedent is given by the Connes--Haagerup classification of injective factors \cite{Connes76, Haagerup87}, which in part involved \cite{Connes73, Connes75} and was then followed by a classification of amenable group actions on said factors \cite{Connes77, Jones80, Ocneanu85, SutherlandTakesaki89, KawahigashiSutherlandTakesaki92, KatayamaSutherlandTakesaki98, Masuda07, Masuda13}.
It has also been recognized early that amenability of the group is necessary to have a manageable classification theory of this sort \cite{Jones83}.

In stark comparison to the factor case, simple nuclear \cstar-algebras are not automatically well-behaved, which may stem from too high-dimensional topological information being encoded in their structure \cite{Villadsen99, Rordam03, Toms08, Tikuisis12}.
The Toms--Winter conjecture \cite{ElliottToms08, Winter10, WinterZacharias10, Winter12}, which may by now almost be called a theorem \cite{Rordam04Z, MatuiSato12acta, MatuiSato14UHF, Tikuisis14, SatoWhiteWinter15, BBSTWW, CETWW, CastillejosEvington19}, is postulating that various (a priori) different concepts of being well-behaved all coincide for (non-elementary) separable simple nuclear \cstar-algebras.
A particularly prominent condition on a \cstar-algebra $A$ to be well-behaved is to ask that it shall be $\CZ$-stable, i.e., isomorphic to the tensor product $A\otimes\CZ$. 
Here $\CZ$ is the standard notation for the so-called Jiang--Su algebra \cite{JiangSu99}, which for all intents and purposes may be regarded as an infinite-dimensional \cstar-algebra analog of the complex numbers $\IC$.
The naturality of this condition as a precursor to being classifiable becomes immediate once one realizes that no reasonably computable invariant can (under some very mild conditions) distinguish $A$ from $A\otimes\CZ$.
Additionally, all counterexamples to the initial Elliott conjecture \cite{Elliott94} fail to be $\CZ$-stable.
At a technical level, the process of tensorially stabilizing a simple nuclear \cstar-algebra $A$ with $\CZ$ can be seen as flattening out the structure of a given \cstar-algebra away from the topologically infinite-dimensional, providing enough noncommutative space for performing certain manipulations of elements in $A$ useful for classification.

As a combination of a remarkable work by many people involved in the above references, we now know that $\CZ$-stability indeed grants access to classification by the Elliott invariant for those separable unital simple nuclear \cstar-algebras that satisfy the universal coefficient theorem \cite{RosenbergSchochet87}.
Although the general non-unital case is not yet fully worked out as of publishing this article, it appears to be not too far behind.

The overarching question in the context of this paper is to what extent it should be possible to classify actions of an amenable group $\Gamma$ up to cocycle conjugacy on classifiable \cstar-algebras $A$, by using computable invariants.
There are some satisfactory results available for special groups $\Gamma$ acting sufficiently outerly on special \cstar-algebras, but they are usually underpinned by methods that are related to the choice of the group, such as the Rokhlin property \cite{HermanJones82, HermanOcneanu84, BratteliKishimotoRordamStormer93, Kishimoto95, BratteliEvansKishimoto95, Kishimoto96, EvansKishimoto97, Kishimoto98, Kishimoto98II, ElliottEvansKishimoto98, BratteliKishimoto00, Nakamura00, Izumi04, Izumi04II}.
Since I wish to mostly focus on actions of general amenable groups here, we shall at this point not review the full history of the available classification results for single automorphisms, finite group actions, etc.
If $\alpha: \Gamma\curvearrowright A$ is an action, then it is similarly (as above) the case that it is indistinguishable from the action $\alpha\otimes\id_\CZ:\Gamma\curvearrowright A$ with regard to most invariants one may come up with for $\Gamma$-\cstar-dynamical systems.
This certainly includes for example the induced action on traces, or $\Gamma$-equivariant Kasparov theory \cite{Kasparov88}.
Like in the ordinary classification program, this idea leads one to the concept of equivariant Jiang--Su stability:

\begin{defie}[{see \cite{TomsWinter07, Szabo18ssa}}] 
Let $\Gamma$ be a countable group and $\alpha: \Gamma\curvearrowright A$ an action on a separable \cstar-algebra.
Let $\CD$ be a strongly self-absorbing \cstar-algebra.\footnote{For example, think $\CD=\CZ$ or $\CD=\CO_\infty$ for the purpose of this introduction.}
We say that $\alpha$ is equivariantly $\CD$-stable if it is cocycle conjugate to the action $\alpha\otimes\id_\CD: \Gamma\curvearrowright A\otimes\CD$.
\end{defie}

In conclusion, if one is interested in developing any kind of classification program à la Elliott for (perhaps a certain abstractly determined class of) $\Gamma$-actions on classifiable \cstar-algebras, one also needs to settle whether such actions are automatically equivariantly Jiang--Su stable or not.

The traceless case is taken care of in \cite{Szabo18kp}, where one automatically gets equivariant absorption of the Cuntz algebra $\CO_\infty$.
For more special choices of the acting group this has again been known before, but I would go overboard by listing all those references here.
Regarding amenable group actions on a finite classifiable \cstar-algebra $A$, there exist promising results by Matui--Sato \cite{MatuiSato12, MatuiSato14}, Sato \cite{Sato16}, and Gardella--Hirshberg \cite{GardellaHirshberg18}, which at present require a somewhat restrictive assumption on the tracial state space of $A$ and the induced group action on it.
Not unlike in the factor case, all results of this type tend to fail for actions of non-amenable groups; see \cite{GardellaLupini18}.
The available work on equivariant $\CZ$-stability leads me to the following conjecture.

\begin{conjecturei} \label{conjecture-A}
Let $\Gamma$ be a countable amenable group.
Let $A$ be a separable simple nuclear and $\CZ$-stable \cstar-algebra.
Then every $\Gamma$-action on $A$ is equivariantly $\CZ$-stable.
\end{conjecturei}

Similarly to parts of the literature involving the Toms--Winter conjecture \cite{MatuiSato12acta}, Matui--Sato's notion of (equivariant) property (SI)\footnote{This expression abbreviates ``small isometries''.} is a seminal ingredient to verify equivariant $\CZ$-stability in an abstract context such as above.
Their brilliant insight led to the first satisfactory abstract results yielding equivariant Jiang--Su stability for strongly outer actions of amenable groups in \cite{MatuiSato12, MatuiSato14}.
In a recent work \cite{Sato16}, Sato has further improved on said approach, and showed that all $\Gamma$-actions on a unital \cstar-algebra $A$ as in \autoref{conjecture-A} automatically have property (SI) as long as the extreme boundary of the tracial state space is compact and finite-dimensional.
He then used this to show that such actions (with some further assumptions) are indeed equivariantly $\CZ$-stable by verifying a $\Gamma$-equivariant tracial McDuff-type property, which is there referred to as ``$\Gamma$-equivariant large embeddability of cones''.

The present work has two motivations:
Firstly, a very detailed inspection of Sato's approach reveals that the aforementioned assumption on the tracial state space is not necessary for obtaining the equivariant property (SI).
Secondly (and more importantly), the current state of the literature only allows one to apply property (SI) to unital \cstar-algebras, or perhaps algebraically simple ones with some amount of direct modification; see for example \cite[Section 5]{Nawata16}.
In the context of non-equivariant property (SI) and its applications to the Toms--Winter conjecture, this is not so problematic because the verification of ordinary \cstar-algebraic $\CZ$-stability or finite nuclear dimension may be performed at the level of some carefully chosen hereditary subalgebra, as is thorougly explained in \cite{CastillejosEvington19}.
In stark contrast, if one is investigating the structure of trace-scaling $\Gamma$-actions on stable \cstar-algebras, then one can quickly observe that there may be no $\Gamma$-invariant hereditary subalgebras with (non-trivial) bounded traces, which means that the current iterations of (equivariant) property (SI) are no longer applicable.
The notable exceptions in the literature, where interesting structural results could be proved for trace-scaling automorphisms on stable \cstar-algebras \cite{EvansKishimoto97, ElliottEvansKishimoto98, BratteliKishimoto00, Nawata19}, rely on comparably ad-hoc methods.
Here we therefore redevelop the entire framework of property (SI) from the ground up in the language of possibly non-unital \cstar-algebras that are allowed to possess unbounded traces.

Since this is ultimately a rather technical endeavour, we shall not discuss the general case here but instead sketch the non-unital framework only in a special case. 
Namely, let $A$ be a separable simple nuclear \cstar-algebra with a densely defined, unbounded lower semi-continuous trace $\tau$ which is unique up to scalar multiple.
Then $\tau$ admits a canonical extension to a generalized limit trace $\tau^\omega$ on the ultrapower $A_\omega$; see the preliminary section.
Even though this trace is clearly unbounded, one can simply pick some positive element $a\in A$ with $\tau(a)=1$, and induce a tracial state on the central sequence algebra $A_\omega\cap A'$ via $x\mapsto\tau^\omega(xa)$.
It is immediate that this assignment vanishes on the annihilator $A_\omega\cap A^\perp$, so it descends to a tracial state on Kirchberg's central sequence algebra $F_\omega(A)=(A_\omega\cap A')/(A_\omega\cap A^\perp)$; see \cite{Kirchberg04}.
This gives rise to a notion of a trace-kernel ideal in $F_\omega(A)$.
In complete analogy to the unital case, property (SI) means that a natural comparison property should hold inside the central sequence algebra: If $e,f\in F_\omega(A)$ are two positive contractions such that $e$ is in the trace kernel ideal and $f$ satisfies $\inf_{k\geq 1} \tau^\omega(af^k)>0$, then $e$ is Murray--von Neumann equivalent to an element on which $f$ acts as a unit.
Given also an action $\Gamma\curvearrowright A$, one similarly defines equivariant property (SI) to mean the same statement inside the fixed point algebra $F_\omega(A)^\Gamma$.
We remark that the approach presented in this paper does \emph{not} recover the treatment of property (SI) for order zero maps as required in \cite{BBSTWW, CETWW, CastillejosEvington19}.
Although I find it likely that this could be done with some further modifications, it seems that this additional layer of generality is only needed for applications related to nuclear dimension, which at this point would appear to be exhausted by the aforementioned references.

By a well-known argument due to Matui--Sato \cite{MatuiSato12acta}, property (SI) allows one to perturb a tracially large order zero map $M_k\to F_\omega(A)$ into a unital $*$-homomorphism from a prime dimension drop algebra $Z_{k,k+1}\to F_\omega(A)$ by appealing to R{\o}rdam--Winter's  universal description \cite{RordamWinter10} of $Z_{k,k+1}$.
This provides a sufficient criterion for $A\cong A\otimes\CZ$; see \cite{Kirchberg04, TomsWinter07}.
Using instead the equivariant version of property (SI), an entirely analogous method can be applied to maps going into the fixed point algebra $F_\omega(A)^\Gamma$, yielding equivariant Jiang--Su stability via \cite{Szabo18ssa}.

The main result of this paper is as follows, and may be seen as a generalization of Sato's observation \cite{Sato16} that equivariant property (SI) tends to hold automatically in the appropriate context; see \autoref{cor:equivariant-SI}:

\begin{theoremi} \label{theorem-B}
Let $A$ be a non-elementary separable simple nuclear \cstar-algebra with very weak comparison\footnote{This is defined in the preliminary section. As the name suggests, it is a very weak version of the more well-known strict comparison, so it certainly holds when $A$ is $\CZ$-stable \cite{Rordam04Z}.}.
Let $\Gamma$ be a countable amenable group.
Then $A$ has equivariant property (SI) relative to every action $\Gamma\curvearrowright A$.
\end{theoremi}

Unsurprisingly, the proof has to retrace a lot of Matui--Sato's previous arguments \cite{MatuiSato12acta} and adapt them to the non-unital context.
In order to obtain equivariant property (SI), our approach slightly deviates at the technical level from Sato's recent work \cite{Sato16}.
Namely, we observe a slightly stronger version of a known precursor commonly referred to as ``excision in small central sequences'' early on, which allows us to more directly perform a certain averaging argument from \cite{Sato16} to strengthen the ordinary property (SI) in $F_\omega(A)$ to the equivariant (SI) inside $F_\omega(A)^\Gamma$.
In particular this foregoes the more technical concept of property (TI).
In fact all of this works in a more general context where one considers the inclusion of a certain \cstar-algebra $A$ inside an ultraproduct $\CB_\omega=\prod_{n\to\omega} B_n$ of simple \cstar-algebras with strict comparison, and automatically gets relative versions of (equivariant) property (SI); see \autoref{thm:ordinary-SI} and \autoref{thm:equi-SI} for more details.

In the last section, we exhibit some non-trivial applications of \autoref{theorem-B}.
The first application verifies \autoref{conjecture-A} for \cstar-algebras with finitely many rays of extremal traces, and may be understood as an equivariant and non-unital version of Matui--Sato's main result in \cite{MatuiSato12acta}; see \autoref{thm:equivariant-Z-stability}:

\begin{theoremi} \label{theorem-C}
Let $A$ be a non-elementary separable simple nuclear \cstar-algebra with very weak comparison.
Suppose that $A$ is finite and has finitely many rays of extremal traces.
Let $\Gamma$ be a countable amenable group.
Then every action $\alpha:\Gamma\curvearrowright A$ is equivariantly $\CZ$-stable.
\end{theoremi}

We remark that the last section begins with a short proof of the fact from \cite{Szabo18kp} that amenable group actions on Kirchberg algebras are equivariantly $\CO_\infty$-stable, which utilizes a relative version of \autoref{theorem-B} with $A$ being traceless.
Regarding \autoref{theorem-C}, the other noteworthy ingredient next to \autoref{theorem-B} entering here is that we cannot directly define a tracial ultrapower of $A$ as in the unital case, and instead have to exploit Kirchberg's \cstar-to-$\mathrm{W}^*$ ultrapower construction \cite{Kirchberg94, AndoKirchberg16} to define a suitable notion of tracial central sequence algebra with respect to unbounded traces.
This is needed to relate the existence of a tracially full order zero map $M_k\to F_\omega(A)$ to the behavior of the weak closure of $A$ under the GNS representation of an unbounded trace, and is hence an important ingredient to access the structure theory of injective factors and group actions on them for tackling a \cstar-dynamical challenge.

For \autoref{theorem-C}, it may be useful to remark that the possibility to consider \cstar-algebras above with $A\cong A\otimes\CK$ allows us to observe that \autoref{theorem-C} holds verbatim for cocycle actions as well; see \autoref{rem:cocycle-actions}.
Our second application concerns strongly outer actions on \cstar-algebras $A$ like above.
We obtain a direct generalization of \cite[Theorem D]{Szabo18ssa4}, which implies that strongly outer actions of certain torsion-free elementary amenable groups automatically absorb a model action on the Jiang--Su algebra; see \autoref{cor:extending-ssa4}.
As a further application, we can deduce a generalization of a classification result of Nawata \cite{Nawata19} for strongly outer automorphisms on the (stabilized) Razak--Jacelon algebra \cite{Jacelon13}; see \autoref{cor:generalize-Nawata}:

\begin{theoremi}
Let $A$ be a separable simple nuclear finite $\CZ$-stable \cstar-algebra.
Suppose that $KK(A,A)=0$ and that $A$ has finitely many rays of extremal traces.
Then every strongly outer $\IZ$-action on $A$ has the Rokhlin property. 
Furthermore, two strongly outer actions $\alpha,\beta: \IZ\curvearrowright A$ are cocycle conjugate if and only if there is an affine homeomorphism $\kappa$ on $(T(A),\Sigma_A)$ such that $\kappa(\tau\circ\alpha)=\kappa(\tau)\circ\beta$ for all $\tau\in T(A)$.
\end{theoremi}

The obvious question for further research becomes to what extent there is always a tracially large order zero embedding $M_k\to F_\omega(A)^\Gamma$ for an action $\Gamma\curvearrowright A$ on a classifiable \cstar-algebra, regardless of the precise structure of extremal traces of $A$.
Since this is a somewhat separate problem, it is beyond the scope of the present work, but it is tempting to speculate that equivariant versions of certain new techniques emerging in \cite{CETWW, CastillejosEvington19} may shed light on it and enable a satisfactory approach toward solving \autoref{conjecture-A} in the future.

I would like to thank the referee for making numerous suggestions that improved the readability of this article.


\section{Preliminaries}

Throughout the paper it will be assumed that the reader is familiar with the basic theory of weights on \cstar-algebras, including the generalized GNS construction; see \cite[II.6.7]{Blackadar}.
We follow the convention that a \emph{proper} weight is a non-zero densely defined and lower semi-continuous weight, and a \emph{trace} on a \cstar-algebra is assumed to be a weight satisfying the tracial condition; see \cite[Section 6.1]{Dixmier}.
We also assume familiarity with the cone of lower semi-continuous traces $T(A)$ on a \cstar-algebra $A$\footnote{In particular, the word ``trace'' may refer to possibly unbounded traces here.
Unlike most of the contemporary literature concerning unital \cstar-algebras, the expression $T(A)$ does not denote the tracial states, but all traces on $A$.}; see \cite{ElliottRobertSantiago11} for a detailed exposition.
We will often make use of the Pedersen ideal $\CP(A)$ of a \cstar-algebra $A$ and its basic properties \cite[Section 5.6]{Pedersen} without explicit reference.
Throughout we will say that a \cstar-algebra is non-elementary, if it is not isomorphic to the algebra of compact operators on some Hilbert space.
At least for the last section we assume that the reader is familiar with c.p.c.\ order zero maps between \cstar-algebras; see \cite{WinterZacharias09}.
We will throughout denote by $\omega$ a free ultrafilter on $\IN$.
When it comes to dealing with arguments related to ultraproducts, we will frequently use the standard technique called the ``$\eps$-test''; see \cite[Lemma 3.1]{KirchbergRordam14}.
The details of this will sometimes be omitted, but given in full in selected non-obvious instances.
When $a$ and $b$ are some elements in a \cstar-algebra, we will usually write $a=_\eps b$ as short-hand for $\|a-b\|\leq\eps$.
We will moreover use the short-hand expression that a mathematical statement holds ``for $\omega$-all $n$'' if the set of all $n\in\IN$ for which the statement is true belongs to $\omega$.

\begin{notation}
Let $A$ be a simple \cstar-algebra.
In the context of considering the tracial cone $T(A)$, we use the symbol $0$ for the zero trace, and the symbol $\infty$ for the trace taking the value $\infty$ on all non-zero positive elements; these are the trivial traces on $A$.
Recall that a non-trivial lower semi-continuous trace $\tau$ on $A$ is automatically a faithful proper trace, which we will use without further mention.
We consider the corresponding dimension function $d_\tau: A_+\to [0,\infty]$ given by $d_\tau(a)=\lim_{n\to\infty} \tau(f_n(a))$, where $f_n: [0,\infty)\to [0,1]$ is any pointwise-increasing sequence of continuous functions with $f_n(0)=0$ and $\lim_{n\to\infty} f_n(t)=1$ for all $t>0$. 

In this paper a compact generator $K$ of $T(A)$ is a compact subset $K\subset T(A)\setminus\set{0,\infty}$ such that $\IR^{>0}\cdot K = T(A)\setminus\set{0,\infty}$.\footnote{In particular $K$ may for example be a compact base of $T(A)$, but we make no demand on the obvious map $\IR^{>0}\times K\to T(A)\setminus\set{0,\infty}$ to be injective.}
Recall that a trace $\tau\in T(A)\setminus\set{0,\infty}$ is called extremal, if for all $\tau'\in T(A)$, $\tau'\leq\tau$ implies $\tau'=c\tau$ for some $c\leq 1$.
It is well-known that a proper trace $\tau$ is extremal precisely when its GNS representation is factorial, i.e., $\pi_\tau(A)''$ is a factor; see \cite[Theorem 6.7.3]{Dixmier}.
We say that $A$ has finitely many rays of extremal traces, if there are only finitely many extremal traces modulo scalar multiples. (The term ``ray'' refers to the set of all positive scalar multiples of a given trace.)
On the other hand, we call $A$ traceless, if there are no non-trivial traces on $A$, in which case we explicitly declare the empty set to be a compact generator of $T(A)$.
\end{notation}

\begin{definition}
A simple \cstar-algebra $A$ is said to have very weak comparison, if for any compact generator $K\subset T(A)$, the following is true.
For every $\eps>0$, there exists $\delta>0$ such that
whenever two non-zero positive elements $a,b\in \mathcal P(A)_+$ in the Pedersen ideal satisfy 
\[
\max_{\tau\in K} d_\tau(a)\leq\delta \quad\text{and}\quad \min_{\tau\in K} d_\tau(b) \geq \eps,
\]
then there exists an element $r\in A$ such that $\|r^*br-a\|\leq\eps$.
\end{definition}

\begin{remark} \label{rem:very-weak-comparison}
If $K_1$ and $K_2$ are both compact generators of $T(A)$, then clearly there is some constant $C>0$ such that $K_1\subseteq [C^{-1},C]\cdot K_2$ and $K_2\subseteq [C^{-1},C]\cdot K_1$.
This implies in particular that the property defined above does not depend on a choice of a compact generator.
In the special case where $A$ is unital, the canonical choice for $K$ is the set of all tracial states on $A$, and the definition above then turns out to be equivalent to the property singled out in \cite[Remark 2.6]{MatuiSato12acta}.
If $A$ is simple and traceless, then $A$ has very weak comparison if and only if $A$ is purely infinite, which in turn is the case if and only if $A$ has strict comparison in the sense below.
\end{remark}

\begin{definition}
A simple \cstar-algebra $A$ is said to have local strict comparison\footnote{The better-known property called strict comparison calls for the stabilization $A\otimes\CK$ to have local strict comparison.}, if the following is true.
Whenever two non-zero positive elements $a,b\in \mathcal P(A)_+$ in the Pedersen ideal satisfy $d_\tau(a) < d_\tau(b)$ for all $\tau\in T(A)\setminus\set{0,\infty}$, then it follows that $a\precsim b$, i.e., there is a sequence $r_n\in A$ with $r_n^*br_n\to a$.
\end{definition}

\begin{notation}[cf.\ \cite{Kirchberg04}]
For a specified sequence of \cstar-algebras $B_n$ and a free ultrafilter $\omega$ on $\IN$, we will denote by
\[
\CB_\omega = \prod_{n\in\IN} B_n / \set{ (b_n)_n \mid \lim_{n\to\omega} \|b_n\|=0 }
\]
their ultraproduct \cstar-algebra.
If $\Gamma$ is a discrete group and $\beta_n:\Gamma\curvearrowright B_n$ is a sequence of actions, we furthermore denote by $\beta_\omega:\Gamma\curvearrowright\CB_\omega$ the induced ultraproduct action.

For an inclusion of a (usually separable) \cstar-algebra $A\subset\CB_\omega$, we write
\[
\CB_\omega\cap A' = \set{ x\in\CB_\omega \mid [x,A]=0 },\quad \CB_\omega\cap A^\perp = \set{x\in\CB_\omega \mid xA=Ax=0},
\]
and define
\[
F(A,\CB_\omega) = (\CB_\omega\cap A')/(\CB_\omega\cap A^\perp).
\]
If furthermore $A$ is $\beta_\omega$-invariant, then so are $\CB_\omega\cap A'$ and $\CB_\omega\cap A^\perp$, and we get an induced action on the quotient $\tilde{\beta}_\omega:\Gamma\curvearrowright F(A,\CB_\omega)$.
In the special case where $B_n=A$ for all $n$ and the inclusion $A\subset A_\omega$ is the obvious one, we abbreviate $F_\omega(A)=F(A,A_\omega)$.
\end{notation}

\begin{definition}
We say that an inclusion of \cstar-algebras $A\subseteq B$ is simple, if $\overline{BaB}=\overline{BAB}$ for every non-zero element $a\in A$.
\end{definition}

\begin{remark} \label{rem:traces-simple-inclusion}
Let $A\subseteq B$ be a simple inclusion of \cstar-algebras.
If $\tau$ is a lower semi-continuous trace on $B$, then it is immediate that the restriction $\tau|_A$ is a non-trivial trace if and only if it is a proper faithful trace on $A$.
This follows from the fact that if $\tau|_A$ is non-trivial, then $0<\tau(a)<\infty$ for some $a\in\CP(A)_+$, but then $\tau$ must take values in $(0,\infty)$ on all of $\CP(A)_+\setminus\set{0}$ because every element in the latter set generates the same algebraic ideal (in $B$) as $a$; see \autoref{prop:trace-ineq} below for a related argument.
\end{remark}

\begin{remark}
Clearly an inclusion $A\subseteq B$ is simple if either $A$ or $B$ is simple.
It is also simple if the inclusion is full in the sense that every non-zero element in $A$ generates $B$ as an ideal.
However, an inclusion can be simple without being full:
an obvious example is an inclusion of the form $A\cong A\oplus 0\subseteq A\oplus A$ for a simple \cstar-algebra $A$.
Perhaps a bit less obvious, but pertinent to the main body of the paper, is the fact that the ultrapower $B_\omega$ of a \cstar-algebra $B$ cannot have any full elements if $B$ is simple, stable and admits a non-trivial trace.
Nevertheless the inclusion $B\subseteq B_\omega$ is simple.
\end{remark}


\section{Traces on ultraproducts and relative commutants}

\begin{definition} \label{def:generalized-limit-trace}
Let $B_n$ be a sequence of \cstar-algebras and $\omega$ a free ultrafilter on $\mathbb N$.
For a sequence $\tau_n$ of lower semi-continuous traces on $B_n$, we may define a lower semi-continuous trace $\tau^\omega: \Big( \prod_{n\in\mathbb N} B_n \Big)_+\to [0,\infty]$ via
\[
\tau^\omega( (b_n)_n ) = \sup_{\eps>0} \lim_{n\to\omega} \tau_n\big( (b_n-\eps)_+ \big),
\]
where $b_n\in A$ is a bounded sequence of positive elements.\footnote{See \cite[Lemma 3.1]{ElliottRobertSantiago11}, which implies that this map indeed yields a lower semi-continuous trace.}
We see that $\tau^\omega( (b_n)_n ) = 0$ whenever $\lim_{n\to\omega} \|b_n\| = 0$, so $\tau^\omega$ induces a lower semi-continuous trace on $\CB_\omega$, which we also denote by $\tau^\omega$ with slight abuse of notation.
A trace of this form on $\CB_\omega$ shall be called a generalized limit trace.
\end{definition}

\begin{notation}
For convenience, we will write $\tau^\eps(b)=\tau\big( (b-\eps)_+ \big)$ for a trace $\tau$ on a \cstar-algebra $B$, a positive element $b\in B$, and some $\eps>0$.
Under this notation, a generalized limit trace as above is of the form
\[
\tau^\omega( (b_n)_n ) = \sup_{\eps>0} \lim_{n\to\omega} \tau_n^\eps(b_n).
\]
\end{notation}

\begin{remark} \label{rem:traces-on-FA}
With the same assumptions as in \autoref{def:generalized-limit-trace}, let additionally $A\subseteq \CB_\omega$ be a \cstar-subalgebra.
Suppose that $\tau^\omega$ is a generalized limit trace on $\CB_\omega$.
Let $a\in A_+$ be a positive element.
The assignment $\tau^\omega_a: (\CB_\omega\cap A')_+\to [0,\infty]$, $x\mapsto\tau^\omega(xa)$ defines a tracial weight satisfying $\tau^\omega_a(x)\leq \tau^\omega(a)\|x\|$ for all $x\in (\CB_\omega\cap A')_+$.
By inserting as $x$ a positive contraction that acts as a unit on $a$, we see that there are three cases.
Firstly, this tracial weight is trivial when $\tau^\omega(a)=0$.
Secondly, we may have $\tau^\omega(a)=\infty$ and then the trace $\tau^\omega_a$ is unbounded.
Thirdly, if $0<\tau^\omega(a)<\infty$, then $\tau^\omega_a$ extends to a positive tracial functional on $\CB_\omega\cap A'$ with norm  $\|\tau^\omega_a\|=\tau^\omega(a)$.

Furthermore, it is clear that $\tau^\omega_a(x)=0$ for all $x\in \CB_\omega\cap A^\perp$.
This induces a trace on $F(A, \CB_\omega)$, which we again denote $\tau^\omega_a$, and which has the same norm if it is bounded.
\end{remark}

\begin{proposition} \label{prop:trace-ineq}
Let $B_n$ be a sequence of \cstar-algebras with ultraproduct $\CB_\omega$ and let $A\subset\CB_\omega$ be an inclusion of a separable simple \cstar-subalgebra.
Let $a,b\in\mathcal P(A)$ be two non-zero positive elements in the Pedersen ideal.
Then there exists a constant $C>0$ such that for every generalized limit trace $\tau^\omega$ on $\CB_\omega$, we have
\[
C^{-1}\tau^\omega_a \leq \tau^\omega_b \leq C\tau^\omega_a \quad\text{on } (\CB_\omega\cap A')_+.
\]
Consequently the same inequality holds on $F(A, \CB_\omega)_+$.
\end{proposition}
\begin{proof}
The two elements $a,b\in\CP(A)$ necessarily generate the same algebraic ideal. 
We may thus choose $M\geq 1$ and elements 
\[
x_1,\dots,x_M, y_1,\dots,y_M\in A
\] 
such that
\[
b=\sum_{j=1}^M x_jax_j^* \quad\text{and}\quad a=\sum_{j=1}^M y_jby_j^*.
\]
We have for any $z\in (\CB_\omega\cap A')_+$ that
\[
\begin{array}{ccl}
\tau^\omega_b(z) &=& \tau^\omega(bz) \\
&=& \dst \sum_{j=1}^M \tau^\omega(x_jax_j^*z) \\
&=& \dst \sum_{j=1}^M \tau^\omega(a^{1/2}x_j^*x_ja^{1/2}z) \\
&\leq& \dst \tau^\omega_a(z)\cdot\Big(\sum_{j=1}^M \|x_j\|^2\Big).
\end{array}
\]
Exchanging the roles of $a$ and $b$ we get the analogous inequality
\[
\tau_a^\omega(z)\leq\tau_b^\omega(z)\cdot\Big(\sum_{j=1}^M \|y_j\|^2\Big).
\]
Hence the constant $C=\max\set{\sum_{j=1}^M \|x_j\|^2, \sum_{j=1}^M \|y_j\|^2}$ does the trick.
\end{proof}

\begin{definition} \label{def:null-full}
Let $B_n$ be a sequence of simple \cstar-algebras with ultraproduct $\CB_\omega$, and let $A\subset \CB_\omega$ be a simple inclusion of a non-zero separable \cstar-subalgebra.
\begin{enumerate}[label=\textup{(\roman*)},leftmargin=*]
\item \label{def:null-full:1} 
We say that a positive contraction $f\in \CB_\omega\cap A'$ or $f\in F(A, \CB_\omega)$ is tracially supported at 1, if either one of the following is true:
Either $B_n$ is traceless for $\omega$-all $n$, in which case we demand $\|fa\|=\|a\|$ for all $a\in A_+$.
Or, $B_n$ has non-trivial traces for $\omega$-all $n$, in which case we demand the following:
For every non-zero positive element $a\in\CP(A)$, there exists a constant $\kappa=\kappa(f,a)>0$ such that for every generalized limit trace $\tau^\omega$ with $\tau^\omega|_A$ non-trivial, one has $\displaystyle \inf_{k\in\mathbb N} \tau^\omega_a(f^k) \geq \kappa\tau^\omega(a)$.
\item \label{def:null-full:2} 
We say that a positive element $e\in \CB_\omega\cap A'$ or $e\in F(A, \CB_\omega)$ is tracially null, if $\tau^\omega_a(e)=0$ for every positive element $a\in\CP(A)$ and every generalized limit trace $\tau^\omega$ on $\CB_\omega$ with $\tau^\omega(a)<\infty$.
\end{enumerate}
In light of \autoref{rem:traces-simple-inclusion}, the condition that $\tau^\omega|_A$ be non-trivial is equivalent to $0<\tau^\omega(a)<\infty$, regardless of the choice of $a$.
Due to \autoref{prop:trace-ineq}, if we ask for $A$ to be simple, then it suffices to check either one of the conditions above for an arbitrary single element $a\in\CP(A)_+\setminus\set{0}$.

\end{definition}

\begin{remark}
The above term ``tracially supported at 1'' stems from the fact that every tracial state of the form $\tau^\omega_a$ on $\CB_\omega\cap A'$ in particular restricts to a positive linear functional on $\cstar(f)$.
By virtue of functional calculus, the latter can be identified with the continuous functions on the spectrum of $f$ vanishing at zero, with $f$ representing the identity function.
Since the restriction of $\tau^\omega_a$ to this \cstar-algebra must come from a finite regular Borel measure on the spectrum of $f$, the condition $\inf_{k\geq 1} \tau_a^\omega(f^k)>0$ means that such a measure is supported on the subset $\set{1}$ inside the spectrum of $f$.
In the above definition, we ask that this happens uniformly, meaning that all measures on the spectrum of $f$ coming from all possible restrictions of tracial states of the form $\tau^\omega_a$ take a value on $\set{1}$ that is uniformly bounded below from a positive constant.
\end{remark}

The next two definitions have their origin in \cite{Sato10, MatuiSato12, MatuiSato12acta, MatuiSato14}:

\begin{definition} \label{def:property-SI}
Let $B_n$ be a sequence of simple \cstar-algebras with ultraproduct $\CB_\omega$.
Suppose that $\Gamma$ is a countable discrete group and $\beta_n: \Gamma\curvearrowright B_n$ is a sequence of actions giving rise to the ultraproduct action $\beta_\omega: \Gamma\curvearrowright\CB_\omega$.
Let $A\subset\CB_\omega$ be a simple inclusion of a (non-zero) separable and $\beta_\omega$-invariant \cstar-subalgebra.
We say that the inclusion $A \subset \CB_\omega$ has equivariant property (SI) relative to $\beta_\omega$ if the following holds:

Whenever $e,f\in F(A,\CB_\omega)^{\tilde{\beta}_\omega}$ are two positive contractions  such that $f$ is tracially supported at 1 and $e$ is tracially null, there exists a contraction $s\in F(A,\CB_\omega)^{\tilde{\beta}_\omega}$ with $fs=s$ and $s^*s=e$.

In particular, we define the inclusion $A\subset\CB_\omega$ to have property (SI) if the above holds for $\Gamma=\set{1}$.
\end{definition}

\begin{notation}
More specifically, we say that a separable simple \cstar-algebra $A$ has property (SI) relative to an action $\alpha: \Gamma\curvearrowright A$, if the canonical inclusion $A\subset A_\omega$ has property (SI) relative to the ultrapower action $\alpha_\omega$.
If moreover $\Gamma=\set{1}$, we just say that $A$ has property (SI); cf.\ \cite[Definition 4.1]{MatuiSato12acta}.
\end{notation}

\begin{remark}
From this point onward, a number of results to come generalize but closely follow Matui--Sato's original approach towards proving property (SI) for simple nuclear \cstar-algebras with strict comparison.
For those familiar with their work we shall briefly summarize the parallels regarding the key arguments here:
\begin{itemize}[leftmargin=*]
\item \autoref{prop:trace-ineq} corresponds to \cite[Lemma 2.4]{MatuiSato12acta}.
\item \autoref{lemma:null-full-mod} corresponds to \cite[Lemma 2.3]{MatuiSato12acta}.
\item \autoref{lem:subequivalence} corresponds to \cite[Lemma 2.5]{MatuiSato12acta}.
\item \autoref{lem:excision} is a slightly more general version of \cite[Section 3]{MatuiSato12acta}.
\item \autoref{thm:ordinary-SI} corresponds to the beginning of \cite[Section 4]{MatuiSato12acta}.
At the level of proof we closely follow \cite[Section 4]{BBSTWW}.
\item \autoref{thm:equi-SI} generalizes \cite[Proposition 5.1]{Sato16}.
\item \autoref{thm:equivariant-Z-stability} generalizes \cite[Theorem 4.9]{MatuiSato14}.
\end{itemize}
We may in a few places also use ideas originating in \cite{Sato12, TomsWhiteWinter15, KirchbergRordam14, BBSTWW}.
\end{remark}

The following technical lemma is exclusively interesting in the non-unital case, as will be explained in the remark after it.
In a nutshell, it will tell us that for a simple inclusion $A\subset\CB_\omega$ as in \autoref{def:null-full}, one can isolate some sequence of compact generators $K_n\subset T(B_n)$ such that, up to scalar multiple, all generalized limit traces $\tau^\omega$ with $\tau^\omega|_A$ non-trivial arise from a sequence $\tau_n\in K_n$.
This will in turn make it possible later on to apply the $\eps$-test within crucial steps in the main proofs, in a similar fashion as has been done in the known cases.

\begin{lemma} \label{lem:compact-bases}
Let $B_n$ be a sequence of simple \cstar-algebras with ultraproduct $\CB_\omega$.
Let $A\subset \CB_\omega$ be a simple inclusion of a separable \cstar-subalgebra.
Suppose that $a\in \mathcal P(A)$ is a positive element in the Pedersen ideal with norm one.
\begin{enumerate}[label=\textup{(\roman*)},leftmargin=*]
\item  \label{lem:compact-bases:1}
There exists a sequence $a_n\in \mathcal P(B_n)$ of positive elements in the Pedersen ideal  representing $a$ and a natural number $m=m_a$ such that

\[
\lim_{n\to\omega} \tau_n(a_n)\leq m\tau^\omega\big( (a-\delta)_+ )
\] 
for all $0\leq\delta<1/2$ and for every sequence $\tau_n$ of traces on $B_n$ giving rise to the generalized limit trace $\tau^\omega$.
\end{enumerate}
Furthermore, there exists a sequence of compact generators $K_n\subset T(B_n)$ with the following properties:
\begin{enumerate}[label=\textup{(\roman*)},leftmargin=*,resume]
\item  \label{lem:compact-bases:2}
Whenever a sequence $\tau_n\in T(B_n)$ of traces gives rise to a generalized limit trace $\tau^\omega$ with $\tau^\omega(a)=1$, then $\tau_n\in K_n$ for $\omega$-all $n\in\mathbb N$.
\item  \label{lem:compact-bases:3}
Every sequence $\tau_n\in K_n$ gives rise to a generalized limit trace $\tau^\omega$ with $\frac{1}{m+1}\leq \tau^\omega(a)\leq m+1$.
\item \label{lem:compact-bases:4}
Whenever $b\in\CP(A)$ is a non-zero positive element and $\tau^\omega$ is a generalized limit trace with $0<\tau^\omega(b)<\infty$, there is a sequence $\theta_n\in K_n$ such that the associated generalized limit trace $\theta^\omega$ on $\CB_\omega$ is a scalar multiple of $\tau^\omega$.
Moreover there exists a constant $C\geq 1$ such that for every such generalized limit trace $\theta^\omega$, one has $C^{-1}\leq\theta^\omega(b)\leq C$.
\end{enumerate}
\end{lemma}
\begin{proof}
\ref{lem:compact-bases:1}:
Let $a_n'\in B_n$ be any sequence of positive contractions representing $a$.
Without loss of generality let us assume that $\|a_n'\|=1$.
As the inclusion $A\subset\CB_\omega$ is simple and we have also assumed $a\in\mathcal P(A)$, it follows that $a$ and $(a-\frac12)_+$ generate the same algebraic ideal in $\CB_\omega$.
Choose contractions $x_1,\dots,x_m\in \CB_\omega$ such that $a=\sum_{j=1}^m x_j(a-\frac12)_+x_j^*$.
Choose sequences of contractions $x_{j,n}\in B_n$ representing $x_j$ for $j=1,\dots,m$.
Set $a_n=\sum_{j=1}^m x_{j,n}(a_n'-\frac12)_+x_{j,n}^* \in \CP(B_n)$, which is another bounded sequence of positive elements representing $a$.
We claim that $m_a=m$ does the trick.

Indeed, for every sequence of traces $\tau_n\in T(B_n)$ one has
\[
\begin{array}{ccl}
\dst \lim_{n\to\omega} \tau_n(a_n) &=& \dst \lim_{n\to\omega} \sum_{j=1}^m \tau_n(x_{j,n}(a_n'-\frac12)_+x_{j,n}^*) \\
&\leq& \dst m\lim_{n\to\omega} \tau_n( (a_n'-\frac12)_+ ) \\
&\leq& \dst m\tau^\omega((a-\delta)_+)
\end{array}
\]
for all $0\leq\delta<1/2$.

For the next part of the statement, we set
\[
K_n = \big\{ \tau\in T(B_n) \mid \frac{m}{m+1} \leq \tau( a_n ) \leq m+1 \big\}
\]
We claim that this sequence does the job.
As each $B_n$ is simple and $a_n$ is in the Pedersen ideal, clearly $K_n\subset T(B_n)$ is a compact generator; cf.\ \cite[Proposition 3.11]{ElliottRobertSantiago11}.

\ref{lem:compact-bases:2}:
Suppose that a sequence $\tau_n\in T(B_n)$ of traces induces a generalized limit trace with $\tau^\omega(a)=1$.
Then by part \ref{lem:compact-bases:1} we have 
\[
1=\tau^\omega(a) = \sup_{\eps>0} \lim_{n\to\omega} \tau_n( (a_n-\eps)_+ ) \leq \lim_{n\to\omega} \tau_n(a_n) \leq m\tau^\omega(a) = m. 
\]
So indeed $\tau_n\in K_n$ for $\omega$-all $n$.

\ref{lem:compact-bases:3}:
For every sequence $\tau_n\in K_n$ and associated generalized limit trace $\tau^\omega$ on $\CB_\omega$, one has
\[
\tau^\omega(a) \leq \lim_{n\to\omega} \tau_n(a_n) \leq m+1
\]
and
\[
\tau^\omega(a) \geq \frac1m \lim_{n\to\omega} \tau_n(a_n) \geq \frac{1}{m+1}.
\]

\ref{lem:compact-bases:4}:
Since we assumed the inclusion $A\subset\CB_\omega$ to be simple, the elements $a,b\in\CP(A)$ generate the same algebraic ideal in $\CB_\omega$.
Hence the claim follows directly from \ref{lem:compact-bases:2} and \ref{lem:compact-bases:3}, with the same argument as in the proof of \autoref{prop:trace-ineq}.
\end{proof}

\begin{remark} \label{rem:constant-bases}
In the lemma above, if we consider the special case where $A$ is simple and $A\subset A_\omega$ is the canonical inclusion into its ultrapower, then we can specify any compact generator $K\subset T(A)$, which will result in the statement \autoref{lem:compact-bases}\ref{lem:compact-bases:4} to hold for $K_n=K$ for all $n\geq 1$.
If $A$ for example happens to be unital, the canonical choice for $K$ would be the set of all tracial states, which would work for the choice of $a=1$ in \autoref{lem:compact-bases}.
\end{remark}

\section{Excision in small central sequences}

\begin{lemma} \label{lemma:null-full-mod}
Let $B_n$ be a sequence of simple \cstar-algebras with ultraproduct $\CB_\omega$.
Let $K_n\subseteq T(B_n)\setminus\set{0,\infty}$ be a sequence of compact generators.
Let $a\in \CB_\omega$ be a positive element of norm one.
\begin{enumerate}[label=\textup{(\roman*)}, leftmargin=*]
\item \label{lemma:null-full-mod:1}
Suppose that there is a constant $\kappa>0$ such that $\tau^\omega(a^k)\geq\kappa$ for every $k\geq 1$ and every generalized limit trace $\tau^\omega$ associated to a sequence $\tau_n\in K_n$.
Then there exist sequences $b_n, c_n\in B_n$ of positive norm one elements satisfying
\[
a=[(b_n)_n],\quad \lim_{n\to\omega} \min_{\tau\in K_n} \tau(c_n) \geq \kappa,
\]
and $b_nc_n=c_n$ for all $n\in\mathbb N$.
\item \label{lemma:null-full-mod:2}
Suppose that for every generalized limit trace $\tau^\omega$ associated to a sequence $\tau_n\in K_n$, one has $\tau^\omega(a)=0$.
Then there exists a sequence of positive contractions $b_n\in B_n$ representing $a$ such that
\[
\lim_{n\to\omega}\max_{\tau\in K_n} d_\tau(b_n)=0.
\]
\end{enumerate}
\end{lemma}
\begin{proof}
Throughout the proof, let $a_n\in B_n$ be a sequence of positive elements with norm one representing $a$. 

\ref{lemma:null-full-mod:1}:
Let $\tau^\omega$ be a generalized limit trace arising from a sequence $\tau_n\in K_n$.
Then $\tau^\omega$ restricts to a lower semi-continuous trace on 
\[
\cstar(a)\cong\set{ h\in\CC(\sigma(a)) \mid h(0)=0 },
\] 
where it corresponds to forming the integral with respect to a $[0,\infty]$-valued regular Borel measure on the (pointed) spectrum of $a$.
By assumption, we have
\[
\inf_{k\in\mathbb N} \tau^\omega(a^k) \geq \kappa,
\]
which by the monotone convergence theorem implies that $\tau^\omega \geq \kappa\cdot\operatorname{ev}_1$.

It follows that for any positive function $g\in \mathcal C_0(0,1]$ with norm one and $g(1)=1$, we have $\tau^\omega(g(a))\geq\kappa$.
Since the choice of the sequence $\tau_n$ was arbitrary, it follows that\footnote{Here we implicitely use the $\eps$-test \cite[Lemma 3.1]{KirchbergRordam14} on the product $\prod_{n\in\IN} K_n$ with respect to the functions $f_{n}^{(m)}: K_n\to [0,1]$, $n,m\geq 1$, defined by $f_n^{(m)}(\tau)=\tau^{1/m}(g(a_n))$.
If the supremum in the claim were less than $\kappa$, we would be able to find a sequence $\tau_n\in K_n$ for which the associated generalized limit trace $\tau^\omega$ satisfies $\tau^\omega(g(a))<\kappa$, a contradiction.}
\[
\sup_{\eps>0} \lim_{n\to\omega} \min_{\tau\in K_n} \tau^\eps(g(a_n)) \geq \kappa.
\]
For $\delta>0$, we let $g_\delta, h_\delta: [0,1]\to [0,1]$ be the functions defined via
\[
h_\delta(t)= \begin{cases} (1-\delta)^{-1}t &,\quad t\leq 1-\delta \\ 1 &,\quad 1-\delta\leq t, \end{cases}
\]
and
\[
g_\delta(t) = \begin{cases} 0 &,\quad t\leq 1-\delta \\ \frac 2\delta (t-\delta) &,\quad 1-\delta \leq t \leq 1-\delta/2 \\ 1 &,\quad 1-\delta/2\leq t. \end{cases}
\]
If we define $b_n^\delta = h_\delta(a_n)$ and $c_n^\delta=g_\delta(a_n)$, then evidently $b_n^\delta c_n^\delta=c_n^\delta$.
Furthermore we have
\[
\lim_{n\to\omega} \min_{\tau\in K_n} \tau( c_n^\delta ) \geq \sup_{\eps>0} \lim_{n\to\omega} \min_{\tau\in K_n} \tau^\eps( c_n^\delta ) \geq \kappa.
\]
by the above, and clearly $\|a_n-b_n^\delta\|\leq \frac{\delta}{1-\delta} \stackrel{\delta\to 0}{\longrightarrow} 0$.
Once we let $\delta\to 0$, the existence of the desired sequences $b_n,c_n\in B_n$ follows from the $\eps$-test \cite[Lemma 3.1]{KirchbergRordam14} as follows.
We set
\[
X_n = \set{ (b,c)\in B_n\times B_n \mid 0 \leq b,c\leq 1,\ \|c\|=1,\ bc=c},\quad n\in\IN.
\]
Consider $X=\prod_{n\in\IN} X_n$ and the functions $f^{(m)}_n: X_n\to [0,\infty]$ defined by $f^{(1)}_n(b,c)=\|a_n-b\|$ and
\[
f^{(1+m)}_n(b,c) = \begin{cases} 0 &,\quad \min_{\tau\in K_n} \tau(c) \geq \kappa-\frac1m \\ 1 &,\quad\text{else}. \end{cases}
\]
For every $\eps>0$, let us consider the sequence $x=(b_n^\delta,c_n^\delta)_{n\in\IN}\in X$ for some $\delta>0$ with $\frac{\delta}{1-\delta}\leq\eps$.
Then $f^{(1)}_\omega(x)=\lim_{n\to\omega} \|a_n-b_n^\delta\|\leq\frac{\delta}{1-\delta}\leq\eps$.
Furthermore, we have $\lim_{n\to\omega}\min_{\tau\in K_n}\tau(c_n^\delta)\geq\kappa$, which for every $m\geq 1$ implies that $\min_{\tau\in K_n} \tau(c_n^\delta)\geq \kappa-\frac1m$ for $\omega$-all $n$.
Thus $f^{(1+m)}_\omega(x)=\lim_{n\to\omega} f^{(1+m)}_n(b_n^\delta,c_n^\delta)=0$ for all $m\geq 1$. 
Applying the $\eps$-test we may get a sequence $x=(b_n,c_n)_{n\in\IN}\in X$ such that
\[
0=f^{(1)}_\omega(x)=\lim_{n\to\omega} \|a_n-b_n\|
\]
and for all $m\geq 1$, one has
\[
0=f^{(1+m)}_\omega(x)=\lim_{n\to\omega} f^{(1+m)}_n(b_n,c_n).
\]
By definition of the functions $f^{(1+m)}_n$, the latter condition translates to $\min_{\tau\in K_n} \tau(c_n)\geq\kappa-\frac1m$ for $\omega$-all $n$.
Since $m\geq 1$ is arbitrary, this means $\lim_{n\to\omega}\min_{\tau\in K_n} \tau(c_n)\geq\kappa$.

\ref{lemma:null-full-mod:2}:
By exploiting the $\eps$-test once more as in the previous footnote, our assumption implies $\dst \lim_{n\to\omega} \max_{\tau\in K_n} \tau^{\eps}(a_n)=0$ for all $\eps>0$.
Let $\eps_\ell>0$ be a monotone null sequence.
We set 
\[
\delta_{n,\ell} = \sqrt{ \max_{\tau\in K_n} \tau^{\eps_\ell}(a_n) },
\] 
which converges to zero as $n\to\omega$, for all $\ell\geq 1$.
By functional calculus, we observe for $b_{n,\ell}=(a_n-(\eps_\ell+\delta_{n,\ell}))_+ \leq a_n$ that
\[
\delta_{n,\ell}(b_{n,\ell}+\frac1m)^{-1}b_{n,\ell} \leq (a_n-\eps_\ell)_+,\quad m\in\mathbb N
\]
and hence
\[
\begin{array}{ccl}
\displaystyle \max_{\tau\in K_n} d_\tau(b_{n,\ell}) &=& \displaystyle \max_{\tau\in K_n} \sup_{m\in\mathbb N} \tau( (b_{n,\ell}+\frac1m)^{-1}b_{n,\ell} ) \\
&\leq& \displaystyle \max_{\tau\in K_n} \delta_{n,\ell}^{-1}\tau^{\eps_\ell}(a_n) \\
&\leq& \displaystyle \delta_{n,\ell} \ \stackrel{n\to\omega}{\longrightarrow} \ 0.
\end{array}
\]
Clearly we also have $\lim_{n\to\omega} \|a_n-b_{n,\ell}\|\leq \eps_\ell$.
The existence of the desired sequence $b_n$ follows directly from applying the $\eps$-test.
\end{proof}

\begin{lemma} \label{lem:subequivalence}
Let $B_n$ be a sequence of simple \cstar-algebras with local strict comparison.
Let $\CB_\omega$ be their ultraproduct, and let $A\subset\CB_\omega$ be a simple inclusion of a non-zero separable \cstar-subalgebra.
Let $e,f\in \CB_\omega\cap A'$ be two positive contractions so that $f$ is tracially supported at 1 and $e$ is tracially null.
Let $a\in A_+$ be a positive element of norm one.
Then there exists a contraction $r\in \CB_\omega$ with
\[
ar=r,\ fr=r, \quad\text{and}\quad r^*r - e \in \CB_\omega\cap A^\perp.
\]
\end{lemma}
\begin{proof}
Since the elements $f$ and $a$ commute by assumption, we note that $ar=r$ and $fr=r$ hold together precisely when $far=r$.
Let $\delta>0$ and $\mathcal F\subset A$ be a self-adjoint finite set.
In order to show the statement, it suffices to find a contraction $r\in \CB_\omega$ with $\|far-r\|\leq \delta$ and $\|(r^*r-e)x\|\leq \delta$ for all $x\in\mathcal F$.

Let $K_n\subset T(B_n)$ be a sequence of compact generators satisfying the property in \autoref{lem:compact-bases}\ref{lem:compact-bases:4}.
By an analogous functional calculus argument as in the proof of \autoref{lemma:null-full-mod}, we may find two non-zero positive elements $a_0,a_1\in\mathcal P(A)_+$ of norm one in the Pedersen ideal such that $\|a-a_0\|\leq\delta$ and $a_0a_1=a_1$.
We may furthermore choose an element $h\in \mathcal P(A)_+$ such that $\|hx-x\|\leq\delta$ for all $x\in \mathcal F$.

Let $e_n,f_n\in B_n$ be sequences of positive contractions representing $e$ and $f$, respectively.
Let $a_{0,n}, h_n\in B_n$ be sequences of positive contractions representing $a_0$ and $h$, respectively.

Since $f$ is tracially supported at 1 and the compact generators $K_n\subset T(B_n)$ satisfy the property in \autoref{lem:compact-bases}\ref{lem:compact-bases:4}, there is a constant $\kappa>0$ such that for every generalized limit trace $\tau^\omega$ on $\CB_\omega$ arising from a sequence $\tau_n\in K_n$ and every $k\geq 1$, we have $\tau^\omega(f^ka_1)\geq\kappa$.
In particular, one also has $\kappa\leq \tau^\omega( (fa_0)^k )$ for all $k\geq 1$.
For the special case that $B_n$ is traceless for $\omega$-all $n$, let us also point out that the norm of $fa_0$ is one by assumption.

Using \autoref{lemma:null-full-mod}\ref{lemma:null-full-mod:1} on the element $a_0f\in\CB_\omega$, we may choose sequences of positive elements $b_n, c_n\in B_n$ of norm one, with $b_nc_n=c_n$ and
\[
\lim_{n\to\omega} \| a_{0,n}^{1/2} f_n a_{0,n}^{1/2} - b_n\| =0,\quad  \lim_{n\to\omega} \min_{\tau\in K_n} \tau(c_n) \geq \kappa.
\]
By assumption, $e$ is tracially null in $\CB_\omega\cap A'$.
In particular, we get that for every generalized limit trace $\tau^\omega$ arising from a sequence $\tau_n\in K_n$, we have $\tau^\omega(eh)=0$.
By using \autoref{lemma:null-full-mod}\ref{lemma:null-full-mod:2} on the element $he\in\CB_\omega$, there is hence a sequence of positive contractions $d_n\in B_n$ with
\[
\lim_{n\to\omega} \|h_n^{1/2} e_n h_n^{1/2} - d_n\| = 0,\quad \lim_{n\to\omega} \max_{\tau\in K_n} d_\tau(d_n)=0.
\]
In summary we conclude that
\[
\max_{\tau\in K_n} d_\tau(d_n) < \kappa/2 \quad\text{and}\quad \min_{\tau\in K_n} d_\tau(c_n) \geq \min_{\tau\in K_n} \tau(c_n) > \kappa/2 
\]
for $\omega$-all $n$.
Since by construction, $K_n\subset T(B_n)$ is a compact generator for every $n$, it follows that for $\omega$-all $n$, the inequality $d_\tau(d_n)<d_\tau(c_n)$ holds for all densely defined traces $\tau\in T(B_n)$.

Since we assumed that every \cstar-algebra $B_n$ has local strict comparison, we conclude that $d_n \precsim c_n$ for $\omega$-all $n$.
This allows us to find a (not necessarily bounded) sequence $q_n\in B_n$ such that
\[
\lim_{n\to\omega} \| q_n^*c_nq_n - d_n \| = 0.
\]
The elements $r_n=c_n^{1/2}q_n$ thus define a sequence with $\lim_{n\to\omega}\|r_n\|\leq 1$, inducing a contraction $r\in \CB_\omega$.
Since clearly $r^*r=he$, it follows due to the choice of $h$ that
\[
(r^*r)x = ehx =_{\delta} ex,\quad x\in\CF.
\]
We also observe
\[
b_n r_n =  b_n c_n^{1/2} q_n = c_n^{1/2} q_n = r_n
\]
for all $n$, and hence
\[
far =_{\delta} fa_0r = [ (b_n r_n) ] = [ (r_n) ] = r. 
\]
This finishes the proof.
\end{proof}

\begin{remark} \label{rem:vwc}
We notice that in the above proof, we actually have the stronger statement $\max_{\tau\in K_n} d_\tau(d_n)<\delta$ for $\omega$-all $n$, where $\delta>0$ is an arbitrary constant.
In the situation where $A=B_n$ for all $n$ and $A\subset A_\omega$ is the canonical inclusion, we can choose $K=K_n$ for all $n$; see \autoref{rem:constant-bases}.
In this special case, if we assume only that $A$ has very weak comparison, then this still gives us a sequence $q_n\in B_n$ as above to carry out the rest of the proof to obtain the statement of \autoref{lem:subequivalence}.
\end{remark}

\begin{notation}[cf.\ {\cite[Section 4]{BBSTWW}}] \label{nota:theta-A}
For a given \cstar-algebra $A$, we can choose a representative for every unitary equivalence class of irreducible representations of $A$, and denote by $\theta_A$ the direct sum representation.
Below it will be relevant to consider the case where $\theta_A$ is essential, i.e., the intersection of its range with the compact operators is trivial.
In the special case where $A$ is simple, this is equivalent to saying that $A$ is non-elementary.
\end{notation}

\begin{lemma} \label{lem:nuclear-maps}
Let $A$ be a separable \cstar-algebra such that $\theta_A$ is essential.
Let $B$ be a \cstar-algebra and $\phi: A\to B$ a nuclear completely positive map.
For every $\eps>0$ and finite set $\CF\subset A$, there exist natural numbers $N,L$, pairwise inequivalent pure states $\lambda_1,\dots,\lambda_L$ on $A$ and elements $c_i\in B, d_{i,l}\in A$ for $i=1,\dots,N$ and $l=1,\dots,L$ such that
\[
\phi(x) =_\eps \sum_{l=1}^L\sum_{i,j=1}^N \lambda_l(d_{i,l}^*xd_{j,l})c_i^*c_j
\]
for all $x\in\CF$.
\end{lemma}
\begin{proof}
This is proved in \cite[Lemma 4.8]{BBSTWW}.
Even though the statement there assumes that $A=B$ is unital and that $\phi$ is the identity map, neither one of those assumptions enters in the proof, so we may conclude the more general statement above.
\end{proof}

\begin{lemma} \label{lem:excision}
Let $B_n$ be a sequence of simple \cstar-algebras with local strict comparison.
Let $\CB_\omega$ be their ultraproduct, and let $A\subset\CB_\omega$ be a simple inclusion of a non-zero separable \cstar-subalgebra.
Suppose that the representation $\theta_A$ is essential.
Let $e,f\in \CB_\omega\cap A'$ be two positive contractions such that $e$ is tracially null and $f$ is tracially supported at 1.
Let $\phi: A\to \overline{A\CB_\omega A}\cap\set{e}'$ be a nuclear, completely positive contractive map.

Then, for every $\eps>0$ and finite set $\CF\subset A$, there exists $\delta>0$ and a finite set $\mathcal G\subset A$ such that the following is true.

If $b\in A$ is a positive element of norm one with
\[
\max_{x\in\mathcal G} \|[b,x]\| \leq \delta \quad\text{and}\quad \|bx\|\geq\|x\|-\delta,\quad x\in\CG,
\]
then there exists an element $s\in \CB_\omega$ with $s^*s\leq e+\eps$ such that 
\[
fs=s,\ \|bs-s\|\leq\eps, \text{ and } \max_{x\in\mathcal F} \|s^*xs-\phi(x)e\|\leq\eps.
\]
\end{lemma}
\begin{proof}
We proceed similarly as in the proofs of \cite[Proposition 2.2]{MatuiSato12acta} or \cite[Lemma 4.4]{BBSTWW}.
Let $\CF\subset A$ and $\eps>0$ be given.
We can and will assume that $\CF=\CF^*$ consists of contractions.
Let us first choose a positive contraction $h\in A_+$ such that 
\begin{equation} \label{eq:def-h}
\|hxh-x\|\leq\eps/5,\quad x\in\mathcal F.
\end{equation}
We set $\CF_h = \set{ hxh \mid x\in\CF\cup\set{1} }\subset A$.

By \autoref{lem:nuclear-maps} there exist $L,N\in\IN$, pairwise inequivalent pure states $\lambda_1,\dots,\lambda_L$ on $A$ and elements $c_i\in \big(\overline{A\CB_\omega A}\cap\set{e}'\big) , d_{i,l}\in A$ for $i\leq N$ and $l\leq L$, such that
\begin{equation} \label{eq:phi-approx}
\phi(x)=_{\eps/5} \sum_{l=1}^L \sum_{i,j=1}^N \lambda_l(d_{i,l}^*x d_{j,l})c_i^*c_j
\end{equation}
for all $x\in\mathcal F_{h}$.

Without loss of generality we assume that the elements $c_j$ are contractions.
We set 
\[
\mathcal G' = \set{ d_{i,l}^* x d_{j,k} \mid x\in\mathcal F_{h},\  i,j\leq N,\  l,k\leq L } \subset A. 
\]
and choose $\delta>0$ small enough such that
\begin{equation} \label{eq:def-delta}
\delta + \sqrt{2\delta} \leq \frac{\eps}{5L^2N^2}\Big( 1+ \max\set{\|d_{j,l}\| \mid j\leq N,\ l\leq L} \Big)^{-1}.
\end{equation}
By \cite[Lemma 4.7]{BBSTWW}, we find positive elements $a_1,\dots,a_L\in A$ of norm one with the property
\begin{equation} \label{eq:aap-condition}
\max_{l,k\leq L}\ \max_{x\in\mathcal G'}\ \|a_lxa_k-\delta_{l,k}\lambda_l(x)a_l^2\|\leq \delta.
\end{equation}
We set
\begin{equation} \label{eq:def-CG}
\mathcal G = \set{d_{j,l} \mid j\leq N,\ l\leq L}\cup\set{ a_l \mid l\leq L}\cup\set{h} \subset A.
\end{equation}
Let us show that the pair $(\mathcal G, \delta)$ satisfies the desired property.

%
Let $b\in A_+$ be a positive element of norm one with
\begin{equation} \label{eq:def-b}
\max_{x\in\mathcal G} \|[x,b]\|\leq\delta \quad\text{and}\quad \|bx\|\geq\|x\|-\delta,\quad x\in\CG.
\end{equation}
We may apply \autoref{lem:subequivalence} a total of $L$ times and find contractions $r_1,\dots,r_L\in \CB_\omega$ 
\begin{equation} \label{eq:subeq-1}
fr_l=r_l,\quad ba_l^2b\cdot r_l = \|ba_l^2b\|r_l,\quad l=1,\dots,L
\end{equation}
and
\begin{equation} \label{eq:subeq-2}
r_l^* r_l - e\in  \CB_\omega\cap A^\perp,\quad l=1,\dots,L.
\end{equation}
Note that 
\[
\begin{array}{ccl}
\|(1-b)r_l\|^2 &=& \|r_l^*(1-b)^2r_l\| \\
&\leq& \|r_l^*(1-b^2)r_l\| \\
&\leq& \|r_l^*(1-ba_l^2b)r_l\| \\
&\stackrel{\eqref{eq:subeq-1}}{\leq}& 1-\|ba_l^2b\| \\
&\stackrel{\eqref{eq:def-b},\eqref{eq:def-CG}}{\leq}& 1-(1-\delta)^2 \ \leq \ 2\delta.
\end{array}
\]
We set
\[
s_0 = \sum_{l=1}^L \sum_{j=1}^N d_{j,l} a_l  b r_l c_j.
\]
Then the relation $fs_0=s_0$ follows trivially with \eqref{eq:subeq-1}.
Furthermore, using the upper estimate $\|(1-b)r_l\|^2\leq 2\delta$ from above we have
\[
\begin{array}{ccl}
\|(1-b)s_0\| &\leq& \displaystyle \sum_{l=1}^L \sum_{j=1}^N \Big( \|[d_{j,l},b]\| +  \|d_{j,l}\|\|[a_l,b]\|+  \|d_{j,l}\|\|(1-b)r_l\| \Big) \\
&\leq& \displaystyle \sum_{l=1}^L \sum_{j=1}^N \Big( \|[d_{j,l},b]\| +  \|d_{j,l}\|\|[a_l,b]\|+  \sqrt{2\delta}\|d_{j,l}\| \Big) \\
&\stackrel{\eqref{eq:def-b},\eqref{eq:def-CG}}{\leq}& 4NL (\delta+\sqrt{2\delta})\cdot\max_{j,l} \|d_{j,l}\| \\
&\stackrel{\eqref{eq:def-delta}}{\leq}& 4\eps/5.
\end{array}
\]
Moreover, we have for all $x\in\mathcal F_h$ that
\[
\begin{array}{ccl}
s_0^*xs_0 &=& \displaystyle \sum_{k,l=1}^L \sum_{i,j=1}^N c_i^* r_l^* b a_l d_{i,l}^* x d_{j,k} a_k b r_k c_j \\
&\stackrel{\eqref{eq:aap-condition},\eqref{eq:def-delta}}{=}_{\makebox[0pt]{\footnotesize\hspace{-7mm}$\eps/5$}}& \displaystyle \sum_{l=1}^L \sum_{i,j=1}^N c_i^* r_l^*  b a_l \lambda_l(d_{i,l}^* x d_{j,l})  a_l b  r_l c_j \\
&\stackrel{\eqref{eq:subeq-1}}{=}& 
\displaystyle \sum_{l=1}^L \sum_{i,j=1}^N \|ba_l^2b\| \lambda_l(d_{i,l}^* x d_{j,l}) c_i^* r_l^*r_l c_j   \\
&\stackrel{\eqref{eq:def-CG},\eqref{eq:def-b}}{=}_{\makebox[0pt]{\footnotesize\hspace{-6mm}$2\eps/5$}}& \displaystyle \sum_{l=1}^L \sum_{i,j=1}^N \lambda_l(d_{i,l}^* x d_{j,l}) c_i^* r_l^*r_l c_j \\
&\stackrel{\eqref{eq:subeq-2}}{=}& \displaystyle \sum_{l=1}^L \sum_{i,j=1}^N \lambda_l(d_{i,l}^* x d_{j,l}) c_i^* c_j e \\
&\stackrel{\eqref{eq:phi-approx}}{=}_{\makebox[0pt]{\footnotesize$\eps/5$}}& \phi(x)e.
\end{array}
\]
Here we have used in the penultimate equality that $c_i\in \big(\overline{A\CB_\omega A}\cap\set{e}'\big)$ for all $i\leq N$.
Finally, we set $s=hs_0$.
For $x=h^2$, the above calculation ensures in particular that 
\[
s^*s=s_0^*h^2s_0 \leq \eps+ \phi(h^2)e \leq \eps+e.
\]
For every $x\in\CF$, it follows by our choice of $h$ and the definition of $\CF_h$ that
\[
s^*xs=s_0^*hxhs_0 =_{4\eps/5} \phi(hxh)e \stackrel{\eqref{eq:def-h}}{=}_{\eps/5} \phi(x)e.
\]
The equation $fs=s$ is inherited from $s_0$.
Using the calculation above involving $b$ and $s_0$, we finally compute
\[
\begin{array}{ccl}
\|(1-b)s\| &\leq& \|s_0\|\|[b,h]\|+\|(1-b)s_0\| \\
&\leq & 4\eps/5+\|s_0\|\|[b,h]\| \\
&\stackrel{\eqref{eq:def-CG},\eqref{eq:def-b}}{\leq}& 4\eps/5+NL\delta\max_{j,l}\|d_{l,j}\|  \ \stackrel{\eqref{eq:def-delta}}{\leq} \ \eps.
\end{array}
\]
This finishes the proof.
\end{proof}

\begin{remark}
In comparison to the existing literature on property (SI) and excision of c.p.c.\ maps in small central sequences, the only noteworthy additional layer of generality above is given by the element $b$.
This aspect is (to my knowledge) irrelevant in the context of ordinary property (SI), but we will see in the next section that it becomes very useful for obtaining equivariant property (SI) with respect to actions of amenable groups.
\end{remark}

\begin{remark} \label{rem:excision-vwc}
In light of \autoref{rem:vwc}, we can again consider the special case $A=B_n$ and the canonical inclusion $A\subset A_\omega$.
Then the statement \autoref{lem:excision} is true when we assume that $A$ is non-elementary separable simple nuclear and has very weak comparison.
\end{remark}

\begin{theorem} \label{thm:ordinary-SI}
Let $B_n$ be a sequence of simple \cstar-algebras with local strict comparison.
Let $\CB_\omega$ be their ultraproduct, and let $A\subset\CB_\omega$ be a simple inclusion of a non-zero separable nuclear \cstar-subalgebra.
Let us also suppose that the representation $\theta_A$ is essential.
Then the inclusion $A\subseteq\CB_\omega$ has property (SI).
\end{theorem}
\begin{proof}
Let $e,f\in\CB_\omega\cap A'$ be two positive contractions such that $e$ is tracially null and $f$ is tracially supported at 1.
We denote by $\tilde{e},\tilde{f}\in F(A,\CB_\omega)$ their induced elements.
Considering the statement of \autoref{lem:excision} for 
\[
\phi=\id_A: A\to A\subset \overline{A\CB_\omega A} \cap \set{e}',
\] 
we may insert as $b$ an approximate unit of $A$, let $\CG\subset A$ get bigger, and let $\delta\to 0$.
This allows us to get the conclusion of \autoref{lem:excision} for arbitrary finite sets $\CF\subset A$ and $\eps>0$, with possibly varying choices for the element $b$.
Once we apply the $\eps$-test, this allows us to find a contraction $s\in\CB_\omega$ such that
\[
fs=s,\quad s^*s\leq e,\quad s^*xs=xe\quad\text{for all } x\in A.
\]
We conclude $s\in\CB_\omega\cap A'$ with the following standard computation.
For all $x\in A$, one has
\[
\begin{array}{ccl}
(xs-sx)^*(xs-sx) &=& s^*x^*xs - x^*s^*xs - s^*x^*sx + x^*s^*sx \\
&=& x^*xe-x^*xe-x^*ex + x^*s^*sx \\
&=& x^*(s^*s-e)x \ \leq \ 0.
\end{array}
\]
In particular, we get $s^*s-e\in\CB_\omega\cap A^\perp$.
For the element $\tilde{s}\in F(A,\CB_\omega)$ induced by $s$, this yields
\[
\tilde{f}\tilde{s}=\tilde{s} \quad\text{and}\quad \tilde{s}^*\tilde{s}=\tilde{e}.
\]
This finishes the proof.
\end{proof}

\begin{corollary}
Let $A$ be a non-elementary separable simple nuclear \cstar-algebra with very weak comparison.
Then $A$ has property (SI).
\end{corollary}
\begin{proof}
This follows from the proof of \autoref{thm:ordinary-SI} and \autoref{rem:excision-vwc}.
\end{proof}


\section{Equivariant property (SI)}

\begin{proposition}[see\ {\cite[Proposition 2.1]{Sato16}}] \label{prop:small-Rt}
Let $A$ be a separable simple \cstar-algebra.
Let $\set{\alpha_i}_{i\in\mathbb N}$ be a countable family of outer automorphisms on $A$.
Then there exists a sequence $b_n\in A$ of norm one positive contractions such that
\[
\lim_{n\to\infty} \|[b_n,a]\|=0,\quad \lim_{n\to\infty} \|b_na\|=\|a\|, \quad a\in A,
\]
and
\[
\lim_{n\to\infty} \|b_n\alpha_i(b_n)\|=0,\quad i\in\mathbb N.
\]
\end{proposition}

\begin{theorem} \label{thm:equi-SI}
Let $\Gamma$ be a countable amenable group.
Let $B_n$ be a sequence of simple \cstar-algebras with local strict comparison, and let $\beta_n: \Gamma\curvearrowright B_n$ be a sequence of $\Gamma$-actions.
Let $\CB_\omega$ be the associated ultraproduct, and $\beta_\omega: \Gamma\curvearrowright\CB_\omega$ the corresponding ultraproduct action. 
Let $A$ be a non-elementary separable simple nuclear \cstar-algebra, and suppose that $A\subset\CB_\omega$ is an inclusion as a $\beta_\omega$-invariant \cstar-subalgebra with the following property: For every $g\in\Gamma$, if the automorphism $\beta_{\omega,g}|_A$ on $A$ is inner, then $\tilde{\beta}_{\omega,g}$ is trivial on $F(A,\CB_\omega)$.
Then the inclusion $A\subset\CB_\omega$ has property (SI) relative to $\beta_\omega$.
\end{theorem}
\begin{proof}
We denote by $\alpha: \Gamma\curvearrowright A$ the action that arises from restricting $\beta_\omega$.
Let $\Lambda$ denote the normal subgroup of all elements $g\in\Gamma$ such that $\alpha_g$ is inner.
By assumption, the action $\tilde{\beta}_\omega:  \Gamma\curvearrowright F(A,\CB_\omega)$ is trivial when restricted to $\Lambda$, which induces an action $\gamma: \Gamma/\Lambda\curvearrowright F(A,\CB_\omega)$ via $\gamma_{g\Lambda}=\tilde{\beta}_{\omega,g}$.
Obviously one has $F(A,\CB_\omega)^{\tilde{\beta}_\omega} = F(A,\CB_\omega)^\gamma$.

Let $\tilde{e},\tilde{f}\in F(A,\CB_\omega)^{\tilde{\beta}_\omega}$ be two positive contractions such that $\tilde{e}$ is tracially null and $\tilde{f}$ is tracially supported at 1.
Consider two positive contractions $e,f\in \CB_\omega\cap A'$ that lift these elements.

Using \autoref{prop:small-Rt}, there exists a sequence $b_n\in A$ of norm one positive contractions such that
\[
\lim_{n\to\infty} \|[b_n,a]\|=0,\quad \lim_{n\to\infty} \|b_na\|=\|a\|,\quad a\in A,
\]
and
\[
\lim_{n\to\infty} \|b_n\alpha_g(b_n)\|=0,\quad g\in\Gamma\setminus\Lambda.
\]
The first two conditions ensure that we are in the position to apply \autoref{lem:excision} to the identity map $\phi=\id_A$ and for $b_n$ in place of $b$ and for arbitrary choices of $\CF\subset A$ and $\eps>0$.
We can hence find contractions $s_n\in\CB_\omega$ such that
\[
fs_n=s_n,\quad \|b_ns_n-s_n\|\to 0,\quad \|e-s_n^*s_n - (e-s_n^*s_n)_+\|\to 0,
\]
and furthermore
\[
 \|s_n^*xs_n-xe\|\to 0,\quad x\in A.
\]
By the asymptotic behavior of the elements $b_n$ it follows that $\|s_n^*\beta_{\omega,g}(s_n)\|\to 0$ for all $g\in\Gamma\setminus\Lambda$.
Applying the $\eps$-test, we can find a contraction $s\in\CB_\omega$ having the property that
\[
fs=s,\quad s^*s\leq e,\quad s^*xs=xe,\quad s^*\beta_{\omega,g}(s)=0,
\]
for all $x\in A$ and $g\in\Gamma\setminus\Lambda$.
Exactly as in the proof of \autoref{thm:ordinary-SI}, we conclude $s\in\CB_\omega\cap A'$ and $s^*s-e\in\CB_\omega\cap A^\perp$.
For the contraction $\tilde{s}_0=s+(\CB_\omega\cap A^\perp)\in F(A,\CB_\omega)$, this means
\[
\tilde{f}\tilde{s}_0=\tilde{s}_0,\quad \tilde{s}_0^*\tilde{s}_0=e,\quad \tilde{s}_0^*\gamma_{g\Lambda}(\tilde{s}_0)=0,\quad g\in\Gamma\setminus\Lambda.
\]
We can now proceed exactly as in the proof of either \cite[Proposition 4.5]{MatuiSato14} or \cite[Proposition 5.1]{Sato16}.
Given a finite set $L\subset\Gamma/\Lambda$, we see that
\[
\tilde{s}_L := |L|^{-1/2}\sum_{g\Lambda\in L} \gamma_{g\Lambda}(\tilde{s}_0)
\]
still satisfies the first two equations above.
However, if we insert a F{\o}lner sequence of $\Gamma/\Lambda$ in place of $L$, we see that the resulting elements $\tilde{s}_L$ are approximately fixed by $\gamma$ on large finite sets of $\Gamma/\Lambda$.
Applying the $\eps$-test once more, we can thus obtain $\tilde{s}\in F(A,\CB_\omega)^\gamma=F(A,\CB_\omega)^{\tilde{\beta}_\omega}$ such that $\tilde{f}\tilde{s}=\tilde{s}$ and $\tilde{s}^*\tilde{s}=e$.
This finishes the proof.
\end{proof}

\begin{corollary} \label{cor:equivariant-SI}
Let $\Gamma$ be a countable amenable group.
Let $A$ be a non-elementary separable simple nuclear \cstar-algebra with very weak comparison.
Then $A$ has equivariant property (SI) relative to every action $\Gamma\curvearrowright A$.
\end{corollary}
\begin{proof}
Proceed exactly as in the proof of \autoref{thm:equi-SI}, but appeal to \autoref{rem:excision-vwc} in place of \autoref{lem:excision}.
Note that the extra assumption about the inclusion is automatic in this case:
If $\alpha$ is an inner automorphism on $A$, then the induced automorphism $\tilde{\alpha}_\omega$ on $F_\omega(A)$ is indeed trivial; see \cite[Remark 1.8]{Szabo18ssa}.
\end{proof}

\begin{remark}
In a sense, the assumption about the inclusion $A\subset\CB_\omega$ appearing in \autoref{thm:equi-SI} may be understood as a a kind of dynamical largeness condition, which can be ensured either by the induced action on $A$ being  outer, or $A$ being sufficiently large in $\CB_\omega$.
Although it is used in an essential way in the proof presented above, I do not know if it is necessary for the conclusion of the theorem to hold.
Apart from the obvious application in the above corollary, there is at least one other potentially interesting case where this assumption holds:

Suppose that $B$ is non-elementary separable simple nuclear with local strict comparison.
Let $\beta: \Gamma\curvearrowright B$ be an action of a countable amenable group.
Suppose that $A\subset F_\omega(B)$ is a unital inclusion of a separable simple unital nuclear and $\tilde{\beta}_\omega$-invariant \cstar-subalgebra.
Denote by $\alpha:\Gamma\curvearrowright A$ the restriction of $\tilde{\beta}_\omega$.
Then we obtain a $\Gamma$-equivariant embedding
\[
(A\otimes B,\alpha\otimes\beta)\to (B_\omega,\beta_\omega) \quad\text{via}\quad a\otimes b\mapsto x_a\cdot b,
\]
where $b\in B$ and $x_a\in B_\omega\cap B'$ satisfies $a=x_a+B_\omega\cap B^\perp$.
Under this inclusion, we can naturally identify $F(A\otimes B,B_\omega)=F_\omega(B)\cap A'$.
If $\alpha_g\otimes\beta_g$ is inner on $A\otimes B$, then $\beta_g$ must also be inner, so $\tilde{\beta}_g$ is trivial on all of $F_\omega(B)$.
This ensures that \autoref{thm:equi-SI} applies to the inclusion $A\otimes B\subset B_\omega$, which gives us a version of equivariant property (SI) inside the relative commutant $F_\omega(B)\cap A'$.
\end{remark}

\begin{definition} \label{def:trace-kernel}
Let $A$ be a non-elementary separable simple \cstar-algebra.
We define the trace-kernel ideal $\CJ_A$ of $F_\omega(A)$ to consist of all those elements $v$ for which $v^*v$ is tracially null in the sense of \autoref{def:null-full} with respect to the canonical inclusion $A\subset A_\omega$.
Evidently $\CJ_A$ is automatically invariant under any automorphism of the form $\tilde{\alpha}_\omega$ on $F_\omega(A)$ induced by an automorphism $\alpha$ on $A$.
\end{definition}

\section{Applications}

In our first immediate application, we show how to recover the traceless version of \autoref{theorem-B} with a short proof, using the methods of this paper:

\begin{corollary}[cf.\ {\cite[Theorem 3.4]{Szabo18kp}}] \label{cor:Kirchberg}
Let $A$ be a Kirchberg algebra and $\Gamma$ a countable amenable group.
Then every $\Gamma$-action on $A$ is equivariantly $\CO_\infty$-stable.
\end{corollary}
\begin{proof}
In light of \autoref{rem:very-weak-comparison}, $A$ is nothing but a traceless separable simple nuclear \cstar-algebras with (local) strict comparison.
Let $\alpha:\Gamma\curvearrowright A$ be an arbitrary action.
We consider the matrix amplification $\alpha^{(2)}=\id\otimes\alpha: \Gamma\curvearrowright M_2\otimes A=M_2(A)$ and the equivariant diagonal inclusion
\[
(A,\alpha) \to (M_2(A)_\omega,\alpha^{(2)}_\omega),\quad a\mapsto \left(\begin{matrix} a & 0 \\ 0 & a \end{matrix}\right).
\]
Then this inclusion evidently satisfies the assumption in \autoref{thm:equi-SI} and does therefore have equivariant property (SI) relative to $\alpha^{(2)}_\omega$.
We borrow an idea from the proof of \cite[Proposition 8.4]{KirchbergRordam02} and apply this fact to the two elements
\[
f=\left(\begin{matrix} 1 & 0 \\ 0 & 0 \end{matrix}\right),\ e=1\ \in \ M_2=M_2(\CM(A))\cap(1\otimes A)' \ \subset \ F( A, M_2(A)_\omega )^{\tilde{\alpha}^{(2)}_\omega}.
\]
Here we implicitly use the canonical isomorphism (see \cite[Proposition 1.5(2)]{BarlakSzabo16})
\[
F( A, M_2(A)_\omega ) \cong M_2(\CM(\overline{A\cdot A_\omega\cdot A}))\cap (1\otimes A)'.
\] 
Recall from \autoref{def:null-full} that since $A$ is traceless, we can consider both elements $e$ and $f$ to be tracially null and tracially supported at 1.
As a consequence of \autoref{thm:equi-SI} and equivariant property (SI) applied to the above data, there exists an isometry $s\in F( A, M_2(A)_\omega )^{\tilde{\alpha}^{(2)}_\omega}$ such that $ss^*\leq f$.
If we represent $s$ by a sequence of contractions $s_n\in M_2(A)$, we may without loss of generality assume that it is of the form $s_n=\left(\begin{matrix} y_n & z_n \\ 0 & 0 \end{matrix}\right)$.
From the fact that $\lim_{n\to\omega} [s_n,1\otimes a] = 0$ for all $a\in A$ we directly get that the sequences $y_n, z_n\in A$ represent two contractions $y,z\in A_\omega\cap A'$.
The fact that $\lim_{n\to\omega} (e-s_n^*s_n)(1\otimes a)$ for all $a\in A$ directly implies that 
\[
1-y^*y,\ 1-z^*z,\ y^*z \ \in \ \tilde{A}_\omega\cap A^\perp.
\]
We also keep in mind that $y-\alpha_{\omega,g}(y)$ and $z-\alpha_{\omega,g}(z)$ are in $A_\omega\cap A^\perp$ for all $g\in\Gamma$.
In conclusion, these two elements induce contractions $\tilde{y}, \tilde{z} \in F_\omega(A)^{\tilde{\alpha}_\omega}$
satisfying the properties $\tilde{y}^*\tilde{y}=1=\tilde{z}^*\tilde{z}$ and $\tilde{y}^*\tilde{z}=0$.
In particular, the fixed point algebra $F_\omega(A)^{\tilde{\alpha}_\omega}$ is properly infinite and we may find a unital inclusion of $\CO_\infty$, which yields the desired conclusion via \cite[Corollary 3.8]{Szabo18ssa}.
\end{proof}

We now focus our attention back to the case of finite \cstar-algebras.

\begin{notation}
If $\rho: A_+\to [0,\infty]$ is a weight on a \cstar-algebra, we denote $D_\rho=\set{x\in A \mid \rho(x^*x)<\infty}$.
\end{notation}

\begin{definition}[cf.\ {\cite[Theorem 7.4.10]{Pedersen}}] \label{def:quasi-invariant}
Let $\rho$ be a proper weight on $A$ and $\alpha$ an automorphism on $A$.
We say that $\rho$ is $\alpha$-quasi-invariant, if under the GNS representation $(\pi_\rho,\CH_\rho)$, there exists a unitary $u_\alpha\in\CB(\CH_\rho)$ such that 
\[
u_\alpha\pi_\rho(a)u_\alpha^* = \pi_\rho(\alpha(a)) 
\]
for all $a\in A$.
In particular, conjugation by $u_\alpha$ defines an automorphism on $\pi_\rho(A)''$ extending the assignment $\pi_\rho(a)\mapsto\pi_\rho(\alpha(a))$ on $\pi_\rho(A)$.
\end{definition}

\begin{proposition} \label{prop:quasi-invariant}
Suppose that $\rho$ is a proper weight and $\alpha$ is an automorphism on $A$.
If there exists a constant $C>0$ such that $C^{-1}\rho \leq \rho\circ\alpha\leq C\rho$, then $\rho$ is $\alpha$-quasi-invariant.
\end{proposition}
\begin{proof}
If $C>0$ is a constant as above, then clearly 
\[
C^{-1} \|\cdot\|_{2,\rho} \leq \|\alpha(\cdot)\|_{2,\rho} \leq C\|\cdot\|_{2,\rho}.
\]
Hence we see that $D_\rho=\alpha(D_\rho)$ in $A$, and $\alpha$ induces a bounded invertible operator $\bar{\alpha}\in \mathcal B(\mathcal H_\rho)$ with $\bar{\alpha}[x]=[\alpha(x)]$ for all $x\in D_\rho$.
We observe
\[
(\bar{\alpha}\pi_\rho(a)\bar{\alpha}^{-1})[x]=(\bar{\alpha}\pi_\rho(a))[\alpha^{-1}(x)] = \bar{\alpha}[a\alpha^{-1}(x)] = [\alpha(a)x]=\pi_\rho(\alpha(a))[x]
\]
for all $a\in A$ and $x\in D_\rho$.
Furthermore
\[
\begin{array}{ccl}
\langle \bar{\alpha}^*\bar{\alpha}\pi_\rho(a)[x] \mid [y] \rangle_{\rho} &=& \langle \bar{\alpha}[ax] \mid \bar{\alpha}[y] \rangle_{\rho} \\
&=& \rho(\alpha(y^*ax)) \\
&=& \langle \bar{\alpha}[x]\mid \bar{\alpha}[a^*y] \rangle_{\rho} \\
&=& \langle \pi_\rho(a)\bar{\alpha}^*\bar{\alpha}[x] \mid [y] \rangle_{\rho}
\end{array}
\]
for all $a\in A$ and $x,y\in D_\rho$.
Hence $|\bar{\alpha}|\in\pi_\rho(A)'$, which allows us to conclude that $u_\alpha=\bar{\alpha}|\bar{\alpha}|^{-1}\in\mathcal B(\mathcal H_\rho)$ is a unitary as in \autoref{def:quasi-invariant}.
\end{proof}

\begin{definition}
Let $\rho$ be a proper weight on $A$ and $\alpha: \Gamma\curvearrowright A$ an action of a discrete group.
We say that $\rho$ is $\alpha$-quasi-invariant, if $\rho$ is $\alpha_g$-quasi-invariant for all $g\in\Gamma$.
In this case, $\alpha$ extends uniquely to an action $\alpha: \Gamma\curvearrowright\pi_\rho(A)''$.
\end{definition}

The following is a straightforward generalization of the more well-known concept of strong outerness for automorphisms on simple unital \cstar-algebras.

\begin{definition}
Let $A$ be a simple \cstar-algebra.
We say that an automorphism $\alpha$ on $A$ is strongly outer, if it is outer, and moreover for every proper $\alpha$-invariant trace $\tau$ on $A$, the induced automorphism of $\alpha$ on $\pi_\tau(A)''$ is outer.
We will call an action $\Gamma\curvearrowright A$ of a discrete group strongly outer, if $\alpha_g$ is strongly outer whenever $g\neq 1$.\footnote{In particular, keep in mind that a single automorphism $\alpha$ corresponds to a $\IZ$-action, but strong outerness for this action means that $\alpha^n$ is a strongly outer automorphism for every $n\neq 0$. This is of course much stronger than to assume that only $\alpha$ is strongly outer.}
\end{definition}

Recall that an automorphism $\alpha$ on a von Neumann algebra $M$ is called properly outer, if for every non-zero projection $p\in M^\alpha$ fixed by $\alpha$, the induced automorphism on $pMp$ is outer.
It is well-known that it suffices to check this condition only for projections $p$ that are also central in $M$.
We may later call an action of a group properly outer, if it consists of properly outer automorphisms except at the neutral group element.

\begin{proposition}  \label{prop:strong-outer}
Let $A$ be a simple \cstar-algebra with an outer automorphism $\alpha$.
Then $\alpha$ is strongly outer if and only if for every $\alpha$-quasi-invariant proper trace $\tau$ on $A$, the unique extension of $\alpha$ on $\pi_\tau(A)''$ is properly outer.
\end{proposition}
\begin{proof}
The ``if'' part is tautological, so let us consider the ``only if'' part.
Suppose that $\alpha$ is strongly outer in the sense of the above definition, and let $\tau$ be an $\alpha$-quasi-invariant proper trace on $A$.
Set $M=\pi_\tau(A)''$.
Let $e\in Z(M)^\alpha$ be a non-zero central projection fixed by $\alpha$.
Then the map $\tau_0=[a\mapsto \tau(e\cdot a)]$ defines another $\alpha$-quasi-invariant proper trace on $A$ since $A$ is simple.
Furthermore there is a natural isomorphism $\pi_{\tau_0}(A)''=eM$ that respects $\alpha$.
If $\tau_0$ is not genuinely $\alpha$-invariant, then evidently $\alpha$ cannot be inner on $eM$ because it moves a trace.
But if $\tau_0$ is genuinely $\alpha$-invariant, then $\alpha$ is outer on $eM$ by assumption.
Since $e$ was arbitrary, we conclude that $\alpha$ is properly outer on $M$.
\end{proof}

\begin{remark} \label{rem:finitely-many-traces}
Let $A$ be a simple \cstar-algebra so that the tracial cone $T(A)\neq\emptyset$ is finite-dimensional.
Let $\tau_1,\dots,\tau_m$ be a family of representatives for the rays of extremal traces in $T(A)$.
Let $\alpha$ be an automorphism on $A$.
For each $j=1,\dots,m$, the trace $\tau_j\circ\alpha$ must again be an extremal trace.
Hence there is a permutation $\sigma: \set{1,\dots,m}\to\set{1,\dots,m}$ and constants $c_1,\dots,c_m>0$ such that $\tau_j\circ\alpha=c_j\cdot\tau_{\sigma(j)}$ for all $j=1,\dots,m$.
We can observe for the trace $\tau=\sum_{j=1}^m\tau_j$ and the constant
\[
C=\max_{1\leq j\leq m} c_j+c_j^{-1}
\]
that we have $C^{-1} \tau \leq \tau\circ\alpha \leq C\tau$, and hence $\tau$ is $\alpha$-quasi-invariant by \autoref{prop:quasi-invariant}.
We denote (cf.\ \cite[Chapter 6]{Dixmier} for the isomorphism below)
\[
M=\pi_\tau(A)'' \cong \pi_{\tau_1}(A)'' \oplus\dots\oplus \pi_{\tau_m}(A)'',
\] 
which is clearly independent of the choice of $\tau$ or $\tau_1,\dots,\tau_m$.

For every proper trace $\tau'$ on $A$, if we write $\tau'=\sum_{i=0}^k \lambda_i\tau_{j_i}$ for $\lambda_i>0$, then the weak closure $\pi_{\tau'}(A)''$ can be realized naturally as the corner of $M$ coming from the support projection for $\tau_{j_1},\dots,\tau_{j_k}$.
The subset $\set{ j_i \mid i=1,\dots,k }\subset\set{1,\dots,m}$ is then $\sigma$-invariant if and only if $\tau'$ is $\alpha$-quasi-invariant.
\end{remark}

For what comes next, we need to set up a suitable replacement of tracial ultrapowers of \cstar-algebras with respect to unbounded traces, in order to apply it in the context of the above remark and to be able to exploit (equivariant) property (SI) to obtain \autoref{theorem-C}.

\begin{remark}
To appreciate the next definition, the reader should recall the notion of Ocneanu ultrapower for a separably acting von Neumann algebra $M$.
One considers the \cstar-subalgebra 
\[
I(M)=\set{ (x_k) \mid x_k\stackrel{k\to\omega}{\longrightarrow} 0 \text{ in the strong-$*$ topology}}\subset\ell^\infty(\IN,M),
\] 
its normalizer algebra 
\[
N(M)=\set{ y\in \ell^\infty(\IN,M) \mid y I(M) + I(M) y \subseteq I(M)},
\] 
and defines the ultrapower as the resulting quotient $M^\omega=N(M)/I(M)$.
Usually the algebra $I(M)$ is defined via a faithful normal state $\rho$ on $M$, as a bounded sequence $x_k\in M$ converges to zero in the strong-$*$ topology precisely when $\|x_k\|_\rho^\#\to 0$. 
$M^\omega$ turns out to be a von Neumann algebra capturing the asymptotic behavior of bounded sequences in $M$; see \cite{AndoHaagerup14} for a very detailed exposition.
\end{remark}

\begin{definition}[cf.\ \cite{Kirchberg94, AndoKirchberg16, GoldbringHartSinclair18}]
Let $A$ be a \cstar-algebra with a distinguished state $\rho$.
Let $\omega$ be a free ultrafilter on $\mathbb N$.
We write
\[
\|a\|_\rho^\# = \Big( \frac12\rho(a^*a+aa^*) \Big)^{1/2}
\]
for all $a\in A$, and
\[
\|x\|_{\rho,\omega}^\#=\lim_{n\to\omega} \|x_n\|_\rho^\# 
\]
for $x\in A_\omega$ and every bounded sequence $x_n\in A$ representing $x$.
We consider the (hereditary) \cstar-subalgebra
\[
\CI_\rho = \set{ x\in A_\omega \mid \|x\|_{\rho,\omega}^\# = 0 } \subseteq A_\omega
\]
and its normalizer
\[
\CN_\rho = \set{ x\in A_\omega \mid x \CI_\rho + \CI_\rho x \subseteq \CI_\rho }.
\]
The statial ultrapower (or \cstar-to-$\mathrm{W}^*$ ultrapower) of the pair $(A,\rho)$ is defined as the quotient
$A_\rho^\omega = \CN_\rho/\CI_\rho$.
\end{definition}

\begin{remark}[see {\cite[Section 4]{Szabo18ssa2} and \cite{Liao16, Liao17}}] \label{rem:sigma-ideals}
Let $A$ be a \cstar-algebra and $\alpha: \Gamma\curvearrowright A$ an action of a discrete group.
Recall that an $\alpha$-invariant ideal $J\subseteq A$ is called a $\Gamma$-$\sigma$-ideal, if for every separable $\alpha$-invariant \cstar-subalgebra $D\subset A$, there exists a positive contraction $e\in (J\cap D')^\alpha$ such that $ec=c=ce$ for all $c\in J\cap D$.
This generalizes Kirchberg's concept of a $\sigma$-ideal, which arises from considering $\Gamma=\set{1}$.

In \cite[Proposition 4.5]{Szabo18ssa2}, it is proved that if $J$ is a $\Gamma$-$\sigma$-ideal, then the induced equivariant quotient map $\pi: (A,\alpha)\to (A/J,\bar{\alpha})$ is strongly locally semi-split, i.e., for every separable $\bar{\alpha}$-invariant \cstar-subalgebra $B\subset A/J$, there is an equivariant c.p.c.\ order zero map $\psi: (B,\bar{\alpha})\to (A,\alpha)$ with $\pi\circ\psi=\id_B$.
\end{remark}

\begin{proposition} \label{prop:technical-stuff}
Let $A$ be a separable \cstar-algebra with a state $\rho$.
Let $\alpha: \Gamma\curvearrowright A$ be an action of a countable discrete group.
Then
\begin{enumerate}[label=\textup{(\roman*)}]
\item $\CI_\rho\subseteq \CN_\rho$ is a $\sigma$-ideal and contains $A_\omega\cap A^\perp$.  \label{prop:technical-stuff:1}
\item If $\rho$ is $\alpha$-quasi-invariant, then $\CI_\rho$ is invariant under $\alpha_\omega$, which induces a $\Gamma$-action $\alpha^\omega$ on $A_\rho^\omega$.  \label{prop:technical-stuff:2}
Moreover $\CI_\rho\subset \CN_\rho$ becomes a $\Gamma$-$\sigma$-ideal.
\item Set $M=\pi_\rho(A)''$. Then the canonical embedding $A_\rho^\omega\to M^\omega$ into the Ocneanu ultrapower of $M$ is an isomorphism.
The same is true for the canonical map $A_\rho^\omega\cap A'\to M^\omega\cap M'$.
If additionally $\rho$ is $\alpha$-quasi-invariant, then these embeddings respect the canonical actions $\alpha^\omega$ induced by $\alpha$ on both sides, respectively.  \label{prop:technical-stuff:3}
\end{enumerate}
\end{proposition}
\begin{proof}
\ref{prop:technical-stuff:1} is \cite[Proposition 4.10]{AndoKirchberg16} and is an easy consequence of the $\eps$-test.
The inclusion $A_\omega\cap A^\perp\subset\CI_\rho$ follows from the fact that every state on $A$ is continuous with regard to the strict topology; see \cite[Lemma 4.12]{AndoKirchberg16}.

\ref{prop:technical-stuff:2}: The first part of the statement is clear.
For the second part, we proceed as follows.
Let $D\subset\CN_\rho$ be a separable $\alpha_\omega$-invariant \cstar-subalgebra.
Since $\CI_\rho$ admits quasicentral approximately $\alpha_\omega$-invariant approximate units (see \cite[Lemma 1.4]{Kasparov88}), there is a sequence $e_n\in\CI_\rho$ such that $\|e_n-\alpha_{\omega,g}(e_n)\|\to 0$ for all $g\in\Gamma$, $\|e_nc-c\|\to 0$ for all $c\in\CI_\rho\cap D$ and $\|[e_n,a]\|\to 0$ for all $a\in D$. 
By applying the $\eps$-test, we can find a single positive contraction $e\in \big( \CI_\rho\cap D' \big)^{\alpha_\omega}$ such that $ec=c$ for all $c\in\CI_\rho\cap D$, which proves the claim.

\ref{prop:technical-stuff:3}: 
This is essentially \cite[Proposition 3.4]{AndoKirchberg16}, but let us give the brief argument here.
The first part of the statement is a consequence of Kaplansky's density theorem.
If an element $x\in M^\omega$ is represented by a bounded sequence $x_n\in M$, then we may find a bounded sequence $a_n\in A$ with $\|a_n-x_n\|_\rho^\# \leq \frac1n$.
So by definition, the element $a\in A_\rho^\omega$ also represents $x\in M^\omega$ under the canonical embedding $A_\rho^\omega \to M^\omega$.
Since $x$ was arbitrary, this embedding is therefore an isomorphism.
For the second part of the statement, we just need to know that the canonical inclusion $A\subseteq A_\rho^\omega$ is sent to the inclusion $A\subseteq M\subseteq M^\omega$ under this isomorphism, which is trivial.
As $A$ is $*$-strongly dense in $M$, one has $M^\omega\cap A' = M^\omega\cap M'$, which is the image of $A_\rho^\omega\cap A'$ under the isomorphism $A_\rho^\omega \to M^\omega$.
Finally, the last part of the statement is trivial because all the involved maps are induced by the canonical inclusion $A\subseteq M$, which is by definition equivariant.
\end{proof}

\begin{remark} \label{rem:from-traces-to-states}
Let $A$ be a \cstar-algebra and $\tau$ a proper trace on $A$.
Let $b_n\in D_\tau$ be a sequence of positive contractions that contains an approximate unit in its range, and assume $\tau(b_n^2)>0$ for all $n\geq 1$.
For any bounded sequence $x_k\in A$, one has that $x_k\to 0$ in the strong operator topology inside $\pi_\tau(A)''$ if and only if $\|x_ky\|_{2,\tau} \to 0$ for all $y\in D_\tau$.
By the properties of the sequence $b_n$, this is true if and only if $\|x_kb_n\|_{2,\tau}\to 0$ for all $n$.
Define a state $\rho$ on $A$ via
\[
\rho(z)=\sum_{n=1}^\infty 2^{-n}\frac{\tau(b_n z b_n)}{\tau(b_n^2)},\quad z\in A.
\] 
For every bounded sequence $x_k\in A$, we have $x_k\to 0$ in the strong-$*$ topology in $\pi_\tau(A)''$ if and only if for all $n$,
\[
0=\lim_{k\to\infty} \big( \|x_kb_n\|_{2,\tau}^2+\|x_k^*b_n\|_{2,\tau}^2 \big) = \lim_{k\to\infty} \tau(b_nx_k^*x_kb_n+b_nx_kx_k^*b_n) 
\]
Thus we see that the strong-$*$ topology on $\pi_\tau(A)''$ is given on bounded sets via the semi-norm $\|\cdot\|_\rho^\#$, and hence $\pi_\tau(A)''\cong\pi_\rho(A)''$.
In particular, whether the above limit is zero for a bounded sequence $x_k\in A$ and for all $n\geq 1$ is independent of the choice of the sequence $b_n\in D_\tau$.
\end{remark}

\begin{example}
If $A$ is unital in the above construction, then $\tau$ is a bounded trace, and one can simply choose $b_n=1_A$.
With that choice one would have $\|\cdot\|_\rho^\#=\|\cdot\|_{2,\tau}$.
\end{example}

\begin{proposition} \label{prop:technical-stuff-trace}
Let $A$ be a separable \cstar-algebra with a state $\rho$.
Suppose that $\rho$ is constructed from a proper trace $\tau$ as in \autoref{rem:from-traces-to-states}.
Then one has $A_\omega\cap A'\subseteq \CN_\rho$ and the restriction of the quotient map $A_\omega\cap A'\to A_\rho^\omega\cap A'$ is onto.
Moreover this map factorizes through $F_\omega(A)$.
\end{proposition}
\begin{proof}
Let $\tau$ and $b_n\in D_\tau$ be as given in \autoref{rem:from-traces-to-states}, and assume that the state $\rho$ is given by the formula there.
As we have already noted above, the \cstar-algebra $\CI_\rho\subset A_\omega$ does not depend on the specific choice of the sequence $b_n$, so let us assume without loss of generality that $b_n\in\CP(A)$, which ensures $d_\tau(b_n)<\infty$.
Under this assumption, the restriction of $\tau$ to the hereditary subalgebra $\overline{b_nAb_n}$ is bounded.
Hence it follows for the generalized limit trace $\tau^\omega$ on $A_\omega$ induced by $\tau$ that in fact
\[
\tau^\omega(b_n a b_n) = \lim_{k\to\omega} \tau(b_n a_k b_n),\quad n\geq 1,
\]
for every positive element $a\in A_\omega$ represented by a bounded sequence of positive elements $a_k\in A$.

Now let $y\in\CI_\rho$ and $x\in A_\omega\cap A'$ be given.
We have to show $xy, yx\in\CI_\rho$.
By definition of $\rho$ and the above observation, the statement $xy\in\CI_\rho$ is equivalent to $\tau^\omega( b_n (y^*x^*xy+xyy^*x^*)b_n)=0$ for all $n\geq 1$.
Using the tracial condition for $\tau^\omega$ we compute for every $n\geq 1$ that
\[
\begin{array}{ccl}
\tau^\omega\big( b_n xy(xy)^* b_n \big) &
=& \tau^\omega\big( b_n  x y y^* x^* b_n \big) \\
&=& \tau^\omega\big( y^*x^*b_n^2xy\big) \\
&=& \tau^\omega\big( y^*b_nx^*xb_ny\big) \\
&\leq & \|x\|^2 \tau^\omega( y^*b_n^2y ) \\
&=& \|x\|^2 \tau^\omega(b_nyy^*b_n) \ = \ 0.
\end{array}
\]
Furthermore we have
\[
\tau^\omega\big( b_n (xy)^*xy b_n \big) \leq \|x\|^2 \tau^\omega( b_ny^*yb_n)=0.
\]
This confirms $xy\in\CI_\rho$.
The proof for $yx\in\CI_\rho$ is completely analogous.
Since $x$ and $y$ were arbitrary, we conclude $A_\omega\cap A'\subset\CN_\rho$.

The rest of the statement follows from the fact that $\CI_\rho$ is a $\sigma$-ideal; see \cite[Proposition 4.13]{AndoKirchberg16} or \cite[Proposition 1.6]{Kirchberg04}.
Let us recall the argument here.
Suppose that $z\in A_\rho^\omega\cap A'$ is an element.
Since it is in particular an element of $A_\rho^\omega$, it may be represented (by definition) as $z=y+\CI_\rho$ for some element $y\in\CN_\rho$.
The assumption that $z$ commutes with $A$ means that $[y,a]\in\CI_\rho$ for all $a\in A$.
Using that $\CI_\rho\subseteq\CN_\rho$ is a $\sigma$-ideal, we may find a positive contraction $e\in\CI_\rho\cap(A\cup\set{y})'$ such that $e[y,a]=[y,a]$ holds for all $a\in A$.
If we set $x=(1-e)y$, then evidently $z=x+\CI_\rho$, but furthermore we get for every $a\in A$ that
\[
[x,a]=[(1-e)y,a] = (1-e)[y,a] = 0.
\]
In other words, one has that $x\in A_\omega\cap A'$ represents $z\in A_\rho^\omega\cap A'$.
Since $z$ was arbitrary, it follows that the restriction of the quotient map $A_\omega\cap A'\to A_\rho^\omega\cap A'$ is indeed surjective.
Lastly, we have the inclusion $A_\omega\cap A^\perp\subseteq\CI_\rho\cap A'$ by \autoref{prop:technical-stuff}\ref{prop:technical-stuff:1}, which implies that the map indeed factors through $F_\omega(A)$.
\end{proof}

\begin{remark} \label{rem:trace-kernel}
Let $A$ be a separable simple \cstar-algebra with finitely many rays of extremal traces.
Let $\tau_1,\dots,\tau_m$ be a set of representatives for each ray of extremal traces, and set $\tau=\sum_{j=1}^m \tau_j$ as in \autoref{rem:finitely-many-traces}.
If $\tau^\omega$ is the associated generalized limit trace on $A_\omega$ and $a\in\CP(A)$ is a non-zero positive element in the Pedersen ideal, then we see that $\tau^\omega(ax)=0$ implies $\tau_0^\omega(ax)=0$ for all positive $x\in A_\omega\cap A'$ and all generalized limit traces $\tau_0^\omega$ with $0<\tau_0^\omega(a)<\infty$.

Let $\rho$ be a state constructed from $\tau$ as in \autoref{rem:from-traces-to-states}.
Comparing \autoref{prop:technical-stuff-trace} with \autoref{def:trace-kernel} yields that we have $(\CI_\rho\cap A')/(\CI_\rho\cap A^\perp) = \CJ_A$, which is the kernel of the natural surjective map $F_\omega(A)\to A_\rho^\omega\cap A'$.
Let us schematically illustrate the role of every intermediate object in all these constructions so far in a commutative diagram:
\[
\xymatrix{
0 \ar[r] & \CI_\rho\cap A^\perp \ar[r] \ar@{^{(}->}[d] & \CI_\rho\cap A' \ar[r] \ar@{^{(}->}[d] & \CJ_A \ar[r] \ar@{^{(}->}[d] & 0 \\
0 \ar[r] & A_\omega\cap A^\perp \ar[r] \ar@{^{(}->}[d] & A_\omega\cap A' \ar[r] \ar@{=}[d] & F_\omega(A) \ar[r] \ar@{->>}[d] & 0 \\
0 \ar[r] & \CI_\rho\cap A' \ar[r] \ar@{^{(}->}[d] & \CN_\rho\cap A' \ar[r] \ar@{^{(}->}[d] & A_\rho^\omega\cap A' \ar@{^{(}->}[d] \ar[r] & 0 \\
0 \ar[r] & \CI_\rho \ar[r] & \CN_\rho \ar[r] & A_\rho^\omega \ar[r] & 0
}
\]
Since $A_\omega\cap A^\perp \subset A_\omega\cap A'$ is $\Gamma$-invariant and $\CI_\rho\cap A' \subset A_\omega\cap A'$ is a $\Gamma$-$\sigma$-ideal with respect to every $\Gamma$-action on $A$ (\autoref{prop:technical-stuff}\ref{prop:technical-stuff:2}), it follows that the trace-kernel ideal $\CJ_A\subset F_\omega(A)$ is also a $\Gamma$-$\sigma$-ideal with respect to every $\Gamma$-action on $A$.
This is a consequence of the following more general fact:
\end{remark}

\begin{proposition}
Let $\Gamma$ be a countable group and $\beta: \Gamma\curvearrowright B$ an action on a \cstar-algebra.
Suppose that $I \subseteq J\subseteq B$ are two $\beta$-invariant ideals in $B$.
If $J$ is a $\Gamma$-$\sigma$-ideal in $B$, then $J/I$ is a $\Gamma$-$\sigma$-ideal in $B/I$.
\end{proposition}
\begin{proof}
Denote by $\bar{\beta}: \Gamma\curvearrowright B/I$ the induced action under the quotient map $\pi: B\to B/I$.
Let $D\subset B/I$ be a separable $\bar{\beta}$-invariant \cstar-subalgebra.
Let $D_0\subset B$ be a separable $\beta$-invariant \cstar-subalgebra such that $\pi(D_0)=D$ and $\pi(J\cap D_0)=\pi(J)\cap D$.
Since $J$ is a $\Gamma$-$\sigma$-ideal, we find $e_0\in (J\cap D_0')^\beta$ such that $e_0c_0=c_0$ for all $c\in J\cap D_0$.
In particular, we have $e:=\pi(e_0)\in (\pi(J)\cap D')^{\bar{\beta}}$ and $ec=c$ for all $c\in \pi(J\cap D_0) = \pi(J)\cap D$.
Since $D$ was arbitrary, this implies that $\pi(J)=J/I \subseteq B/I$ is a $\Gamma$-$\sigma$-ideal.
\end{proof}

\begin{notation}
Below we will use the standard notation $\CR$ for the hyperfinite II$_1$ factor.
For example, it can be defined as the weak closure $\CR=\pi_\tau(M_{2^\infty})''$ of the CAR algebra $M_{2^\infty}=M_2\otimes M_2\otimes\dots$ under the GNS representation associated to its unique tracial state $\tau$.
\end{notation}

\begin{proposition} \label{prop:vN-stuff}
Suppose that $M$ is a von Neumann algebra that is isomorphic to a finite direct sum of hyperfinite type II factors.
Let $\omega$ be a free ultrafilter on $\mathbb N$.
Let $\Gamma$ be a countable amenable group and $\alpha:\Gamma\curvearrowright M$ an action.
\begin{enumerate}[label=\textup{(\roman*)},leftmargin=*]
\item $(M^\omega\cap M')^{\alpha^\omega}$ is a finite direct sum of II$_1$ factors.
In particular, there exists a unital $*$-homomorphism $M_k\to (M^\omega\cap M')^{\alpha^\omega}$ for all $k\geq 2$. 
\label{prop:vN-stuff:1}
\item Suppose $\alpha$ is properly outer.
For every action $\beta: \Gamma\curvearrowright\CR$, there exists a unital equivariant $*$-homomorphism $(\CR,\beta)\to (M^\omega\cap M',\alpha^\omega)$. 
\label{prop:vN-stuff:2}
\end{enumerate}
\end{proposition}
\begin{proof}
\ref{prop:vN-stuff:1}: This can be proved in exactly the same way as in \cite[Lemmas 4.1 and 4.2(1)]{MatuiSato14}.
Matui--Sato's proof concerns the case where every direct summand of $M$ is of type II$_1$, but their proof works just as well in the general case.

\ref{prop:vN-stuff:2}: Notice that the claim passes to finite direct sums of such von Neumann algebras equipped with $\Gamma$-actions.
Hence let us assume without loss of generality that $\Gamma$ acts transitively via $\alpha$ on the minimal central projections of $M$.
Write $M=Me_1\oplus\dots\oplus Me_\ell$, where each element $e_j\in M$ is a minimal central projection.
Then we have a transitive action $\sigma: \Gamma\curvearrowright\set{1,\dots,\ell}$ such that $\alpha_g(e_j)=e_{\sigma_g(j)}$ for all $g\in\Gamma$ and $j=1,\dots,\ell$.
For every $j=1,\dots,\ell$, let $H_j=\set{ h\in\Gamma \mid \sigma_h(j)=j }$.
Note that we have $|\Gamma/H_j|=\ell$ because the map $\Gamma/H_j\to\set{1,\dots,\ell}$ given by $gH_j\mapsto\sigma_g(j)$ is well-defined and injective by definition, and surjective by the assumption that $\sigma$ is transitive.
We see that for every $j=1,\dots,\ell$, the action $\alpha|_{H_j}$ leaves the factor $Me_j$ invariant and is outer by assumption.
By Ocneanu's theorem \cite{Ocneanu85}, the action on this corner equivariantly absorbs the action $(\beta|_{H_j})^{\bar{\otimes}\infty}: H_j\curvearrowright\CR^{\bar{\otimes}\infty}$ up to cocycle conjugacy.
Hence we find a unital equivariant $*$-homomorphism $\psi_j: (\CR,\beta|_{H_j})\to (e_jM^\omega\cap M',\alpha^\omega|_{H_j})$.

By the first part we may pick a partition of unity $1=\sum_{j=1}^\ell p_j$, where the summands are pairwise equivalent projections in $(M^\omega\cap M')^{\alpha^\omega}$.
Without loss of generality, we may assume that they also commute with the range of $\alpha^\omega_g\circ\psi_j$ for all $g\in\Gamma$ and $j=1,\dots,\ell$.

We define $\psi: \CR\to M^\omega\cap M'$ via
\[
\psi(x) =  \sum_{j=1}^\ell   \sum_{\bar{g}\in\Gamma/H_j} e_{\sigma_{\bar{g}}(j)}\cdot p_j \cdot \alpha^\omega_g(\psi_j(\beta_g^{-1}(x)).
\]
We first observe for every $j=1,\dots,\ell$ that if $g_2=g_1h$ for some $h\in H_j$, then we have $\alpha^\omega_{g_2}\circ\psi_j\circ\beta_{g_2}^{-1}=\alpha^\omega_{g_1}\circ\psi_j\circ\beta_{g_1}^{-1}$ since $\psi_j$ is $H_j$-equivariant.
In particular, the above sum is well-defined and $\psi$ becomes a unital $*$-homomorphism.

We compute for all $x\in\CR$ and $g_0\in\Gamma$ that
\[
\begin{array}{ccl}
\psi(\beta_{g_0}(x)) &=& \dst \sum_{j=1}^\ell   \sum_{\bar{g}\in\Gamma/H_j} e_{\sigma_{\bar{g}}(j)}\cdot p_j \cdot \alpha^\omega_g(\psi_j(\beta_{g_0^{-1}g}^{-1}(x)) \\
&=&  \dst \sum_{j=1}^\ell   \sum_{\bar{g}\in\Gamma/H_j} e_{\sigma_{\overline{g_0g}}(j)}\cdot p_j \cdot \alpha^\omega_{g_0g}(\psi_j(\beta_g^{-1}(x)) \\
&=& \dst \alpha^\omega_{g_0}(\psi(x)),
\end{array}
\]
proving that $\psi$ is $\Gamma$-equivariant.
This finishes the proof.
\end{proof}

\begin{theorem} \label{thm:equivariant-Z-stability}
Let $A$ be a non-elementary separable simple nuclear \cstar-algebra with very weak comparison.
Suppose that $A$ is finite and has finitely many rays of extremal traces.
Let $\Gamma$ be a countable amenable group.
Then every action $\alpha:\Gamma\curvearrowright A$ is equivariantly $\CZ$-stable.
\end{theorem}
\begin{proof}
Let $\tau$ be a trace on $A$ chosen as in \autoref{rem:finitely-many-traces}.
It is quasi-invariant with respect to every automorphism on $A$, so in particular $\alpha$-quasi-invariant.
Set $M=\pi_\tau(A)''$, which is a finite direct sum of hyperfinite factors of type II.
Due to \autoref{rem:trace-kernel} and \autoref{prop:technical-stuff}, we see that there is a surjective and equivariant $*$-homomorphism
\[
(F_\omega(A),\tilde{\alpha}_\omega) \to (M^\omega\cap M',\alpha^\omega),
\]
the kernel of which is equal to $\CJ_A$.
Let $k\geq 2$.
By \autoref{prop:vN-stuff}, there is a unital copy $M_k \subset (M^\omega\cap M')^{\alpha^\omega}$.
Since $\CJ_A$ is a $\Gamma$-$\sigma$-ideal, it follows from \autoref{rem:sigma-ideals} that there is a c.p.c.\ order zero map $\phi: M_k\to F_\omega(A)^{\tilde{\alpha}_\omega}$ lifting such an inclusion, in particular satisfying $1-\phi(1_k)\in\CJ_A$.
Invoking \autoref{rem:trace-kernel}, we see that $1-\phi(1_k)$ is tracially null.
On the other hand, $\phi(e_{1,1})$ is tracially supported at 1 for the following reason.
Firstly, since $e_{1,1}$ is a projection, it follows that $\phi(e_{1,1})-\phi(e_{1,1})^m \in \CJ_A$ for all $m\geq 2$, so $\phi(e_{1,1})$ agrees with all its positive powers on all generalized limit traces induced on $F_\omega(A)$.
Secondly, for every $a\in A_+$ and every generalized limit trace $\theta^\omega$ on $A_\omega$ with $\theta^\omega(a)=1$, the map $\theta^\omega_a\circ\phi$ will become a tracial state on $M_k$, which must coincide with the unique one.
Therefore we have $\inf_{m\geq 1} \theta^\omega_a(\phi(e_{1,1})^m)= \theta^\omega_a(\phi(e_{1,1}))=\frac1k$.

Since $A$ has property (SI) relative to $\alpha$ due to \autoref{cor:equivariant-SI}, it follows that there is a contraction $s\in F_\omega(A)^{\tilde{\alpha}_\omega}$ with $\phi(e_{1,1})s=s$ and $s^*s=1-\phi(1_k)$.
By \cite[Proposition 5.1]{RordamWinter10} this gives rise to a unital $*$-homomorphism $Z_{k,k+1}\to F_\omega(A)^{\tilde{\alpha}_\omega}$.
Since $k\geq 2$ was arbitrary, we can find a unital embedding $\CZ\to F_\omega(A)^{\tilde{\alpha}_\omega}$ of the Jiang--Su algebra.
The conclusion follows by \cite[Corollary 3.8]{Szabo18ssa}.
\end{proof}

For our further applications below we need to appeal to the theory of Rokhlin dimension with commuting towers relative to subgroups; see \cite{Szabo18rd} for the relevant results about them used here.

\begin{lemma} \label{lem:dimrokc-2}
Let $A$ be a separable simple nuclear finite $\CZ$-stable \cstar-algebra.
Suppose that $A$ has finitely many rays of extremal traces.
Let $\Gamma$ be a countable amenable group and $H\subset\Gamma$ a normal subgroup with $\Gamma/H\cong\IZ$.
Then for every strongly outer action $\alpha:\Gamma\curvearrowright A$, we have $\dimrokc(\alpha,H)\leq 2$.
\end{lemma}
\begin{proof}
The proof of this is analogous to \cite[Theorem 2.14]{Szabo18ssa4}, but we shall outline the argument and at which point it differs from the original proof.
Let us first denote $F_\omega(A)^H:=F_\omega(A)^{\tilde{\alpha}_\omega|_H}$ and observe that it carries an action $\gamma: \Gamma/H \cong \IZ \curvearrowright F_\omega(A)^H$ via $\gamma_{gH}(x)=\tilde{\alpha}_{g,\omega}(x)$.
The claim $\dimrokc(\alpha,H)\leq 2$ means that, given any $k\geq 2$, there exist three equivariant c.p.c.\ order zero maps 
\[
\kappa_0, \kappa_1, \kappa_2: \big( \CC(\IZ/k\IZ), \text{shift} \big) \to \big( F_\omega(A)^H, \gamma \big)
\]
with $\kappa_0(1)+\kappa_1(1)+\kappa_1(2)=1$.

Let us consider for every $N\geq 1$ the universal unital \cstar-algebra $Z_{N,N+1}^U$ generated by a contraction $v\in Z_{N,N+1}^U$ and the image of a c.p.c.\ order zero map $\psi: M_N\to Z_{N,N+1}^U$ satisfying $1-\psi(1_N)=v^*v$ and $\psi(e_{1,1})v=v$.
Given a unitary representation $\nu: \Gamma\to\CU(M_{N-1})$, we denote by $\delta^\nu: \Gamma\curvearrowright Z_{N,N+1}^U$ the action determined by $\delta^\nu_g(v)=v$ and $\delta^\nu_g\circ\psi=\psi\circ\operatorname{Ad}(1\oplus\nu_g)$ for all $g\in\Gamma$.
It occupied most of \cite[Section 2]{Szabo18rd} to prove that for any $k\geq 2$, one can find unitary representations $\nu$ that are trivial on $H$ and so that there are three approximate order zero maps from $\CC(\IZ/k\IZ)$ into $Z_{N,N+1}^U$ modeling the above properties approximately.
The logical turning point in the proof of \cite[Theorem 2.14]{Szabo18ssa4} was to apply \cite[Lemma 2.9]{Szabo18ssa4}, which says that $(Z_{N,N+1}^U, \delta^\nu)$ equivariantly maps into $(F_\omega(A),\tilde{\alpha}_\omega)$, but under stronger assumptions on $A$ than stated here.

In summary, in order to show our claim, it suffices to generalize \cite[Lemma 2.9]{Szabo18ssa4}, i.e., to prove that for arbitrary $N\geq 2$ and $\nu: \Gamma\to\CU(M_{N-1})$, we can find a unital equivariant $*$-homomorphism from $(Z_{N,N+1}^U, \delta^\nu)$ to $(F_\omega(A),\tilde{\alpha}_\omega)$.
Let $\tau$ be a trace on $A$ chosen as in \autoref{rem:finitely-many-traces}.
It is quasi-invariant with respect to every automorphism on $A$, so in particular $\alpha$-quasi-invariant.
Set $M=\pi_\tau(A)''$, which is a finite direct sum of hyperfinite factors of type II.
Due to \autoref{rem:trace-kernel} and \autoref{prop:technical-stuff}, we see that there is a surjective and equivariant $*$-homomorphism
\[
(F_\omega(A),\tilde{\alpha}_\omega) \to (M^\omega\cap M',\alpha^\omega),
\]
the kernel of which is equal to $\CJ_A$.
By \autoref{prop:vN-stuff}\ref{prop:vN-stuff:2} we may find a unital equivariant $*$-homomorphism $(M_N, \operatorname{Ad}(1\oplus\nu)) \to (M^\omega\cap M', \alpha^\omega)$.
Since $\CJ_A$ is a $\Gamma$-$\sigma$-ideal, it follows from \autoref{rem:sigma-ideals} that there is an equivariant c.p.c.\ order zero map $\phi: (M_N, \operatorname{Ad}(1\oplus\nu)) \to (F_\omega(A),\tilde{\alpha}_\omega)$ lifting such an inclusion, in particular satisfying $1-\phi(1_N)\in\CJ_A$.
With the same justification as in the proof of \autoref{thm:equivariant-Z-stability}, we see that $\phi(e_{1,1})$ is tracially supported at 1.
Since $A$ has property (SI) relative to $\alpha$ due to \autoref{cor:equivariant-SI}, it follows that there is a contraction $s\in F_\omega(A)^{\tilde{\alpha}_\omega}$ with $\phi(e_{1,1})s=s$ and $s^*s=1-\phi(1_N)$.
By definition of $Z_{N,N+1}^U$ this gives rise to a unique unital $*$-homomorphism $\mu: Z_{N,N+1}^U \to F_\omega(A)$ satisfying $\mu\circ\psi=\phi$ and $\mu(v)=s$.
In particular we can see right away that $\mu$ is equivariant with respect to $\delta^\nu$ and $\tilde{\alpha}_\omega$, which shows the claim.
\end{proof}

\begin{notation}[cf.\ {\cite[Definition B]{Szabo18ssa4}}]
We denote by $\FC$ the smallest class of countable groups with $\set{1}\in\FC$, such that $\FC$ is closed under isomorphism, countable directed unions, and extensions by $\IZ$.

Below we denote by $\cc$ the relation of cocycle conjugacy between actions on \cstar-algebras.
\end{notation}

\begin{theorem} \label{thm:extending-ssa4}
Let $A$ be a separable simple nuclear finite \cstar-algebra.
Suppose that $A$ has finitely many rays of extremal traces.
Let $\Gamma$ be a countable amenable group and $H\subset\Gamma$ a normal subgroup such that $\Gamma/H\in \FC$.
Let $\gamma:\Gamma\curvearrowright\CD$ be a semi-strongly self-absorbing action
and let $\alpha:\Gamma\curvearrowright A$ be a strongly outer action. 
If $\alpha|_H \cc (\alpha\otimes\gamma)|_H$, then $\alpha\cc\alpha\otimes\gamma$.
\end{theorem}
\begin{proof}
We proceed similarly as in the proof of \cite[Theorem 3.1]{Szabo18ssa4}.
Let $\FF$ be the class of all countable amenable groups $\Lambda$ such that the conclusion of this theorem holds whenever one has $\Gamma/H\cong\Lambda$ instead of $\Gamma/H\in\FC$.
Evidently the trivial group is in $\FF$, and $\FF$ is closed under extensions.
Moreover it follows directly from \cite[Theorem 5.6(ii)]{Szabo17ssa3} that $\FF$ is closed under countable directed unions.
By the definition of the class $\FC$, it suffices to show $\IZ\in\FF$ in order to obtain $\FC\subseteq\FF$, which will prove the claim.

Under the assumption that $\Gamma/H\cong\IZ$, we have by \autoref{lem:dimrokc-2} that $\dimrokc(\alpha,H)\leq 2$.
By virtue of \cite[Theorem 5.2]{Sato16} (cf.\ \cite[Theorem 1.16]{Szabo18ssa4}), the action $\gamma$ is equivariantly $\CZ$-stable, and hence unitarily regular by \cite[Proposition 2.19]{Szabo18ssa2}.
Thus the claim follows directly from \cite[Theorem A]{Szabo18rd}; see also \cite[Theorem 2.3]{Szabo18ssa4}.
\end{proof}

\begin{corollary} \label{cor:extending-ssa4}
Let $A$ be a separable simple nuclear finite \cstar-algebra.
Suppose that $A$ has finitely many rays of extremal traces.
Let $\CD$ be a finite strongly self-absorbing \cstar-algebra with $A\cong A\otimes\CD$.
Let $\Gamma$ be a countable amenable group in the class $\FC$.
Then an action $\alpha:\Gamma\curvearrowright A$ is strongly outer if and only if $\alpha\cc\alpha\otimes\gamma$ for every action $\gamma:\Gamma\curvearrowright\CD$.
\end{corollary}
\begin{proof}
Since the ``if'' part is trivial, let us consider the ``only if'' part.
Let us assume that $\alpha$ is strongly outer.
By \cite[Theorem C]{Szabo18ssa4}, it suffices to show $\alpha\cc\alpha\otimes\gamma$ for some action $\gamma:\Gamma\curvearrowright\CD$ which is strongly outer and semi-strongly self-absorbing.
But this case follows directly from \autoref{thm:extending-ssa4} applied to the special case $H=\set{1}$.
\end{proof}

\begin{remark} \label{rem:cocycle-actions}
Let us at this point observe that \autoref{cor:Kirchberg}, \autoref{thm:equivariant-Z-stability} and \autoref{cor:extending-ssa4} hold verbatim for cocycle actions $(\alpha,w):\Gamma\curvearrowright A$.
This is because by the Packer--Raeburn stabilization trick \cite{PackerRaeburn89}, there is a genuine action $\beta: \Gamma\curvearrowright A\otimes\CK$ on the compact stabilization which is exterior equivalent to $(\alpha\otimes\id_A,w\otimes 1)$.
The dynamical systems $(F_\omega(A), \tilde{\alpha}_\omega)$ and $(F_\omega(A\otimes\CK),\tilde{\beta}_\omega)$ are then conjugate; see \cite[Proposition 1.9]{BarlakSzabo16} and \cite[Remark 1.8]{Szabo18ssa}.
Since either one of the theorems above apply to the genuine action $\beta$, the results therefore generalize directly.
\end{remark}

The following main result generalizes \cite[Theorems 7.3, 7.4]{Nawata19}, but we note that the main constribution here is in generalizing the Rokhlin-type theorems \cite[Theorems 6.4, 6.7]{Nawata19} with a slightly different approach.

\begin{corollary} \label{cor:generalize-Nawata}
Let $A$ be a separable simple nuclear finite $\CZ$-stable \cstar-algebra.
Suppose that $KK(A,A)=0$ and that $A$ has finitely many rays of extremal traces.
Then every strongly outer $\IZ$-action on $A$ has the Rokhlin property. 
Furthermore, two strongly outer actions $\alpha,\beta: \IZ\curvearrowright A$ are cocycle conjugate if and only if there is an affine homeomorphism $\kappa$ on $(T(A),\Sigma_A)$ such that $\kappa(\tau\circ\alpha)=\kappa(\tau)\circ\beta$ for all $\alpha,\beta\in T(A)$.
\end{corollary}
\begin{proof}
We first observe that by \cite[Theorem A]{CastillejosEvington19}, $A$ has finite nuclear dimension.
In particular $A$ falls within the scope of the classification theory developed in \cite{ElliottGongLinNiu17} (see also \cite[Corollary D]{CastillejosEvington19}), which implies $A\cong A\otimes\CQ$, where $\CQ$ is the universal UHF algebra.

By \cite[Theorem 5.12]{Szabo17R}, it suffices to prove that every strongly outer action $\alpha:\IZ\curvearrowright A$ has the Rokhlin property.
If we apply \autoref{cor:extending-ssa4} to the special case $\Gamma=\IZ$, it follows that every such action $\alpha$ tensorially absorbs some automorphism on $\CQ$ with the Rokhlin property, and will therefore itself have the Rokhlin property.
This finishes the proof.
\end{proof}


\bibliographystyle{gabor}
\bibliography{master}
\end{document}